\documentclass[12pt]{amsart}
\usepackage{txfonts}
\usepackage{amssymb}
\usepackage{eucal}
\usepackage{graphicx}
\usepackage{amsmath}
\usepackage{amscd}
\usepackage[all]{xy}           

\usepackage{amsfonts,latexsym}
\usepackage{xspace}
\usepackage{epsfig}
\usepackage{float}
\usepackage{color}
\usepackage{fancybox}
\usepackage{colordvi}
\usepackage{multicol}
\usepackage{colordvi}
\usepackage{stmaryrd}

\newtheorem{theorem}{Theorem}[section]
\newtheorem{defn}[theorem]{Definition}
\newtheorem{lemma}[theorem]{Lemma}
\newtheorem{coro}[theorem]{Corollary}
\newtheorem{prop-def}{Proposition-Definition}[section]

\newcommand{\nc}{\newcommand}
\newcommand{\delete}[1]{}

\nc{\mlabel}[1]{\label{#1}}  
\nc{\mcite}[1]{\cite{#1}}  
\nc{\mref}[1]{\ref{#1}}  
\nc{\mbibitem}[1]{\bibitem{#1}} 

\delete{
\nc{\mlabel}[1]{\label{#1}  
{\hfill \hspace{1cm}{\bf{{\ }\hfill(#1)}}}}
\nc{\mcite}[1]{\cite{#1}{{\bf{{\ }(#1)}}}}  
\nc{\mref}[1]{\ref{#1}{{\bf{{\ }(#1)}}}}  
\nc{\mbibitem}[1]{\bibitem[\bf #1]{#1}} 
}

\nc{\bfk}{\mathbf{k}}
\nc{\Der}{\mathrm{Der}}
\nc{\Ker}{\mathrm{Ker}}

\begin{document}

\title{3-Lie bialgebras and 3-pre-Lie algebras induced by involutive derivations }

\author{RuiPu  Bai}
\address{College of Mathematics and Information Science,
Hebei University
\\
Key Laboratory of Machine Learning and Computational\\ Intelligence of Hebei Province, Baoding 071002, P.R. China} \email{bairuipu@hbu.edu.cn}

\author{Shuai Hou}
\address{College of Mathematics and Information  Science,
Hebei University, Baoding 071002, China} \email{hshuaisun@163.com}

\author{Chuangchuang Kang}
\address{College of Mathematics and Information  Science,
Hebei University, Baoding 071002, China} \email{kangchuang2016@163.com}

\date{}

\begin{abstract}  In this paper, we study the structure of 3-Lie algebras with involutive derivations.
We prove that if $A$ is an $m$-dimensional 3-Lie algebra
with  an involutive derivation $D$, then there exists a compatible 3-pre-Lie
algebra $(A, \{ ,  , , \}_D)$ such that $A$ is  the sub-adjacent 3-Lie algebra,  and there is a local
 cocycle $3$-Lie bialgebraic structure on  the $2m$-dimensional semi-direct product 3-Lie algebra  $A\ltimes_{ad^*} A^*$, which is associated to the adjoint representation $(A, ad)$.
 By means of involutive derivations,
  the skew-symmetric solution of the 3-Lie classical Yang-Baxter equation in the 3-Lie algebra  $A\ltimes_{ad^*}A^*$,  a class of 3-pre-Lie algebras, and
 eight and ten dimensional local
 cocycle 3-Lie bialgebras  are constructed.
\end{abstract}


\subjclass[2010]{17B05, 17D99.}

\keywords{ 3-Lie algebra; involutive derivation;  local cocycle 3-Lie bialgebra; 3-Lie Yang-Baxter equation.}

\maketitle
\footnote{ Corresponding author: Ruipu Bai, E-mail: bairuipu@hbu.edu.cn.}



\allowdisplaybreaks

\section{Introduction}
In recent years,  quantum 3-Lie algebras \cite{BLG},   3-Lie bialgebras \cite{BGL, BGLZ}, local cocycle 3-Lie bialgebras, 3-pre Lie algebras and  3-Lie classical Yang-Baxter equation \cite{BGS} are provided.
 It is well known that  Lie bialgebra is
the algebraic structure corresponding to a Poisson-Lie group and the classical structure of a
quantized universal enveloping algebra \cite{CP,D}, and  it has a coboundary theory, which
leads to the construction of Lie bialgebras from solutions of the classical Yang-Baxter equation.
  For a 3-Lie algebra $A$, and $r=\sum\limits_i x_i\otimes y_i\in A\otimes A$, the equation
$$[[r, r, r]] = 0$$
is called the  3-Lie classical Yang-Baxter equation (3-Lie CYBE)  \cite{BGS}, where

\vspace{5mm} $ [[r, r, r]] := [r_{12}, r_{13}, r_{14}]+ [r_{12},r_{23},r_{24}]+[r_{13},r_{23},r_{34}] +[r_{14},r_{24},r_{34}]$

\vspace{2mm}\hspace{15mm} $=\sum\limits_{i,j,k}\big([x_i, x_j, x_k]\otimes y_i\otimes y_j\otimes y_k + x_i\otimes [y_i,x_j,x_k]\otimes y_j\otimes y_k $

\vspace{2mm}\hspace{18mm}$+x_i\otimes x_j\otimes [y_i,y_j,x_k]\otimes y_k+ x_i\otimes x_j \otimes x_k\otimes [y_i,y_j,y_k]\big ).$

But from the solutions of 3-Lie classical Yang-Baxter equation, the 3-Lie bialgebra which was introduced in \cite{BGL}, can not be constructed.
For excavating the ideal connections between  the solutions of 3-Lie classical Yang-Baxter equation and 3-Lie bialgebraic structures, authors  in   \cite{BGS}, introduced  a new 3-Lie coproduct $ \triangle$ on a 3-Lie algebra, and the pair $(A, \triangle)$  is called
a local cocycle 3-Lie bialgebra, where  $\triangle=\triangle_1+\triangle_2+\triangle_3$$:A \rightarrow A^{\wedge 3}$  satisfies that
 the dual $(A^*, \triangle^*)$ is a 3-Lie algebra, and

$\Delta_{1}$ is a $1$-cocycle associated to the representation $(A\otimes A\otimes A,ad\otimes 1\otimes 1)$;

$\Delta_{2}$ is a $1$-cocycle associated to the representation $(A\otimes A\otimes A,1\otimes ad\otimes 1)$;

 $\Delta_{3}$ is a $1$-cocycle associated to the representation $(A\otimes A\otimes A,1\otimes 1\otimes ad)$.

For a skew-symmetric  $r=\sum\limits_i x_i\otimes y_i\in A\otimes A$, if $r$ is a solution of the  3-Lie classical Yang-Baxter equation, then $(A, \triangle)$ is  a local cocycle 3-Lie bialgebra, where
for all $x\in A$,
$$\triangle_1(x)=\sum\limits_{ij}[x, x_i,x_j ]\otimes y_j\otimes y_i; \triangle_2(x)= \sum\limits_{ij} y_i\otimes [x, x_i,x_j ]\otimes y_j; \triangle_3(x)=\sum\limits_{ij} y_j\otimes y_i\otimes [x, x_i,x_j ].$$

The question that  how to get  solutions of  3-Lie classical Yang-Baxter equation and how to construct local cocycle 3-Lie bialgebras is a  hard task.
Motivated by the object,  we study the structure of involutive derivations on 3-Lie algebras. And in terms of involutive derivations,  solutions of  3-Lie classical Yang-Baxter equation,
 a class of eight  and ten dimensional  local cocycle 3-Lie bialgebras, and a class of 3-pre-Lie algebras are constructed.

In the following we assume that  all algebras are over  an algebraically closed field  $F$ with characteristic zero,  if $\{x_1, \cdots, x_m\}$ is a basis of $A$, then  $\{x_1^*, \cdots, x_m^*\}$  is the basis of the dual space $A^*$, where  $\langle x_i^*, x_j\rangle=\delta_{ij}$, $1\leq i, j\leq m,$ and
$Z$ is the set of integers.

\section{Involutive derivations and  compatible $3$-pre Lie algebras }

A {\bf 3-Lie algebra } $(A, [ , , ])$ over a field $F$ is a vector space $A$ with a linear  multiplication $[ , , ]: A^{\wedge 3} \rightarrow A$ satisfying that
for all $x_1, x_2, x_3, x_4, x_5\in A,$
$$
[x_1, x_2, [x_3, x_4, x_5]]=[[x_1, x_2, x_3], x_4, x_5]+[x_3, [x_1, x_2, x_4], x_5]+[x_3, x_4, [x_1, x_2, x_5]],
$$
which is usually called {\it the generalized Jacobi identity}, or  {\it Filippov identity}.

The subalgebra generated by  $[x_1, \cdots, x_n]$, $\forall x_1, \cdots, x_n\in A$, is called the derived algebra of  $A$, and is denoted by  $A^1$.

{\bf A derivation $D$}  is an endomorphism of   $A$  satisfying,
\begin{equation}
 D[x_1, x_2, x_3]=[Dx_1, x_2, x_3]+[x_1, Dx_2, x_3]+[x_1, x_2, Dx_3],  \forall x_1, x_2, x_3\in A.
\label{eq:der}
\end{equation}
Further, if $D$ satisfies  $D^2=I_d$ (identity), then $D$ is called an {\bf involutive derivation} on $A$, and $A$ has the decomposition
\begin{equation}
    A=A_1 ~\dot+~ A_{-1},
    \label{eq:23}
\end{equation}
where $A_1=\{v\in A \mid Dv=v\}$ and $ A_{-1}=\{
v\in A \mid Dv=-v\}$ which are abelian subalgebras.

 $Der(A)$ denotes the
derivation algebra of $A$.
For all  $x_1, x_{2}\in A$, the  left multiplications $ad(x_1, x_{2})$: $A \rightarrow A$,
$ad(x_1,  x_{2})(x)=[x_1, x_{2}, x],  ~ \forall  x\in A,$
are derivations which are called inner derivations, and $ad(A)$ denotes all the inner derivations.

{\bf A representation} of $A$ ( or  {\bf an $A$-module}) is a  pair $(V, \rho)$,  where  $V$ is a vector space,
$\rho: A\wedge A\rightarrow End (V)$ is a linear mapping such that for all $x_i\in A, 1\leq i\leq 4,$

\hspace{1.5cm}$[\rho(x_1, x_2), \rho(x_3, x_4)]=\rho(x_1, x_2)\rho(x_3, x_4)-\rho(x_3, x_4)\rho(x_1, x_2)$

\hspace{5cm}$=\rho([x_1, x_2, x_3], x_4)-\rho([x_1, x_2, x_4], x_3),$

\hspace{1.5cm}\vspace{1mm} $\rho([x_1, x_2, x_3], x_4)=\rho(x_1, x_2)\rho(x_3, x_4)+\rho(x_2, x_3)\rho(x_1, x_4)+\rho(x_3, x_1)\rho(x_2, x_4).$

A linear mapping  $T: V\rightarrow A$ is called {\bf an $\wp$-operator}  associated to an $A$-module $(V,\rho)$ if $T$ satisfies
 \begin{equation}
 [Tu,Tv,Tw]=T(\rho(Tu,Tv)w+\rho(Tv,Tw)u+\rho(Tw,Tu)v), ~\forall u,v,w\in V.
 \label{wq:operator}
 \end{equation}

$(A,ad)$ is called  {\bf the adjoint representation} of $A$. The dual representation of $(A,ad)$ is denoted by  $(A^*,ad^*)$, where    $A^*$ is the dual space of $A$, and
 $ad^*:A^{\wedge^2}\rightarrow End A^*$, for all $x_1, x_2, x\in A$ and $\xi\in A^*$,
 \begin{equation}
\langle ad^*{(x_1,x_2)}\xi, x\rangle=-\langle \xi, ad{(x_1,x_2)}x\rangle.
 \label{eq:dualrep}
 \end{equation}

  There is an equivalent description of an $A$-module $(V, \rho)$, that is,
$(B, [ , , ]_{\rho})$ is a 3-Lie algebra, where $B=A\dot+ V$, for all $x_1, x_2, x_3\in A, v\in V$,
$$[x_1, x_2, x_3]_{\rho}=[x_1, x_2, x_3], ~~ [x_1, x_2, v]_{\rho}=\rho(x_1, x_2)v, ~~ [A, V, V]=[V, V, V]=0.$$

For any 3-Lie algebra $A$, we have  {\bf the semi-direct product 3-Lie algebra} $(B, \mu)$, where $ B=A\dot+ A^*$, ~~ $[ , , ]_*: (A\dot+ A^*)^{\wedge 3}\rightarrow A\dot+ A^*,$  for all  $x_i\in A, \xi_{i}\in A^*, 1\leq i\leq 3$,
 \begin{equation}
\mu(x_{1}+\xi_{1},x_{2}+\xi_{2},x_{3}+\xi_{3})=[x_{1},x_{2},x_{3}]+ad^{*}(x_{1},x_{2})\xi_{3}+ad^{*}(x_{2},x_{3})\xi_{1}+
 ad^{*}(x_{3},x_{1})\xi_{2},
 \label{eq:semiprod}
 \end{equation}
which  is denoted by   $A\ltimes _{ad^*} A^*.$

\begin{theorem}\
Let $(A,[\cdot,\cdot,\cdot])$ be a $3$-Lie algebra with an involutive derivation $D$. Then $D$ is an $\wp$-operator of $A$ associated to the adjoint representation $(A, ad)$, and $D$ satisfies
 \begin{equation}
 [Dx,Dy,Dz]=D([Dx,Dy,z]+[Dy,Dz,x]+[Dz,Dx, y]), ~\forall x,y,z\in A.
\end{equation}
\label{thm:DD}
\end{theorem}

\begin{proof}
By Eq \eqref{eq:der}, for all $x, y, z\in A,$

\vspace{2mm}$
D(ad(Dx,Dy)z+ad(Dy,Dz)x+ad(Dz,Dx)y)$

\vspace{2mm}$=D([Dx,Dy,z]+[Dy,Dz, x]+[Dz,Dx, y])$

\vspace{2mm}$=D([Dx,Dy,D^2z]+[D^2x, Dy, Dz]+[Dx, D^2y, Dzy])=[Dx, Dy, Dz].
$
\\The proof is completed. \end{proof}

A 3-pre-Lie algebra $(A, \{ , , \})$ \cite{BGS} is  a vector space $A$ with a 3-ary linear mapping\\ $\{\cdot,\cdot,\cdot\}$: $A\otimes A \otimes A \rightarrow A$, satisfying $ \forall x_1, x_2, x_3\in A,$
  \begin{equation}
  \{x_1, x_2, x_3\}=-\{ x_2, x_1, x_3 \},
  \label{eq:pre1}
  \end{equation}
  \begin{equation}
\{x_{1},x_{2}, \{x_{3},x_{4},x_{5}\}\} =\{\{x_{1},x_{2},x_{3}\}_{c},x_{4},x_{5}\}+\{x_{3},\{x_{1},x_{2},x_{4}\}_{c},x_{5}\}
  \label{eq:pre2}
  \end{equation}
  $$ \hspace{-8mm}+\{x_{3},x_{4},\{x_{1},x_{2},x_{5}\}\},$$
  \begin{equation}
  \{\{x_{1},x_{2},x_{3}\}_{c},x_{4},x_{5}\}=\{x_{1},x_{2},\{x_{3},x_{4},x_{5}\}\}+\{x_{2},x_{3},\{x_{1},x_{4},x_{5}\}\}
  \label{eq:pre3}
  \end{equation}
  $$+\{x_{3},x_{1},\{x_{2},x_{4},x_{5}\}\},$$

  \begin{equation}
   \{x_1, x_2, x_3\}_c=\{x_1, x_2, x_3\}+\{x_2, x_3, x_1\}+\{x_3, x_1, x_2\}.
 \label{eq:cyc}
 \end{equation}

Therefore, if $(A, \{ , , \})$ is a 3-pre-Lie algebra, then $(A, \{ , , ,\}_c)$ is a 3-Lie algebra, which is called {\bf the sub-adjacent 3-Lie algebra of the  3-pre-Lie algebra $(A, \{ , , \})$}, and
 $(A, \{ , , \})$ is called {\bf the compatible 3-pre-Lie algebra of the  3-Lie algebra $(A, \{, , \}_c)$}.

\begin{theorem}
Let $(A,[\cdot,\cdot,\cdot])$ be a $3$-Lie algebra, $D\in Der(A)$ be an involutive derivation. Then
$(A,\{\cdot,\cdot,\cdot\}_D)$ is a $3$-pre-Lie algebra, where
\begin{equation}
\{ x,y,z\}_D= [Dx, Dy, z].
\label{eq:preD1}
\end{equation}
Furthermore,
\begin{equation}
\{ x,y,z\}_D=
\begin{cases}
0,~~~ ~~~~~~ x,y,z\in A_{1}, or ~x,y,z\in A_{-1},\\
[x,y,z],~~~~  x,y\in A_{1}, z\in A_{-1},\\
-[x,y,z], ~~~~ x\in A_1, y\in A_{-1}, z\in A_{-1},\\
[x,y,z],~~~~  x,y\in A_{-1}, z\in A_{1},\\
-[x,y,z],~~~~  x\in A_1, y\in A_{-1}, z\in A_{1}.\\
\end{cases}
\label{eq:preD}
\end{equation}
And $(A,\{\cdot,\cdot,\cdot\}_D)$ is called {\bf the   $3$-pre-Lie algebra associated with the involutive derivation $D$}.

\label{thm:Liepre}
\end{theorem}

\begin{proof} By Theorem \ref{thm:DD},
$D$ is an $\wp$-operator associate to the adjoint representation $(A, ad)$, and for all  $x_i\in A, 1\leq i\leq 5,$

$[[Dx_{1},Dx_{2},x_{3}]+[Dx_{1},x_{2},Dx_{3}]+[x_{1},Dx_{2},Dx_{3}],Dx_{4},x_{5}]=-[[x_{1},x_{2},x_{3}],Dx_{4},x_{5}],$

  $[D[Dx_{1},Dx_{2},Dx_{3}],Dx_{4},x_{5}]=-[D[x_{1},x_{2},x_{3}],Dx_{4},x_{5}].$

By Eq \eqref{eq:preD1},
 $\{ x_1, x_2, x_3\}_D$
 $=[Dx_1, Dx_2, x_3]=-\{ x_2, x_1, x_3\}_D$, Eq \eqref{eq:pre1} holds.
Since

\vspace{2mm}$ [Dx_{3},Dx_{4},[Dx_{1},Dx_{2},x_{5}]]$
$=[Dx_{1},Dx_{2},[Dx_{3},Dx_{4},x_{5}]]-[[Dx_{1},Dx_{2},Dx_{3}],Dx_{4},x_{5},]$

\vspace{2mm}\hspace{5cm}$+[Dx_{3},[Dx_{1},Dx_{2},Dx_{4}],x_{5}]$,

\vspace{2mm}$\{\{x_{1},x_{2},x_{3}\}_{Dc},x_{4},x_{5}\}_{D}+\{x_{3},\{x_{1},x_{2},x_{4}\}_{Dc},x_{5}\}_{D}+\{x_{3},x_{4},\{x_{1},x_{2},x_{5}\}_{D}\}_{D}$

\vspace{2mm}$=2([([Dx_{1},x_{2},x_{3}]+[x_{1},Dx_{2},x_{3}]+[x_{1},x_{2},Dx_{3}]),Dx_{4},x_{5}]$

\vspace{2mm}$+[Dx_{3},([Dx_{1},x_{2},x_{4}]+[x_{1},Dx_{2},x_{4}]+[x_{1},x_{2},Dx_{4}]),x_{5}]$

\vspace{2mm}$+[[Dx_{1},Dx_{2},Dx_{3}],Dx_{4},x_{5}]+[Dx_{3},[Dx_{1},Dx_{2},Dx_{4}],x_{5}])$$+[Dx_{1},Dx_{2},[Dx_{3},Dx_{4},x_{5}]]$

\vspace{2mm}$=2([D[x_{1},x_{2},x_{3}],Dx_{4},x_{5}]+[[Dx_{1},Dx_{2},Dx_{3}],Dx_{4},x_{5}]$

\vspace{2mm}$+[D[x_{3},D[x_{1},x_{2},x_{4}],x_{5}]+[Dx_{3},[Dx_{1},Dx_{2},Dx_{4}],x_{5}])$$+[Dx_{1},Dx_{2},[Dx_{3},Dx_{4},x_{5}]]$

\vspace{2mm}$=[Dx_{1},Dx_{2},[Dx_{3},Dx_{4},x_{5}]],$ we get  Eq \eqref{eq:pre2}.

Therefore,
$A$ is a 3-pre-Lie algebra in the multiplication  \eqref{eq:preD1}.
Eq \eqref{eq:preD} follows from Eqs \eqref{eq:23} and \eqref{eq:preD1}, and a direct computation.
\end{proof}

\begin{theorem}
Let $(A,[\cdot,\cdot,\cdot])$ be a $3$-Lie algebra, $D$ be an involutive derivation on $A$. Then $D$ is an algebra isomorphism from the sub-adjacent 3-Lie algebra $(A, \{\cdot,\cdot,\cdot\}_{Dc} )$ of the 3-pre-Lie algebra $(A, \{\cdot,\cdot,\cdot\}_D)$ to the 3-Lie algebra  $(A,[\cdot,\cdot,\cdot])$,
and
\begin{equation}
\{x,y,z\}_{Dc}=\{x,y,z\}_D+\{y,z,x\}_D+\{z,x,y\}_D=D[Dx,Dy,Dz], ~~ \forall x, y, z\in A.
\label{eq:liepreD}
\end{equation}
\label{thm:isor}
\end{theorem}
Furthermore,
\begin{equation}
\{ x,y,z\}_{Dc}=
\begin{cases}
0,~~~ ~~~~~~~x,y,z\in A_{1}, or ~x,y,z\in A_{-1},\\
-[x,y,z],~~~~ x,y\in A_{1}, z\in A_{-1},\\
-[x,y,z], ~~~~x,y\in A_{-1}, z\in A_{1}.\\
\end{cases}
\label{eq:liepreD'}
\end{equation}

\begin{proof} By Eq \eqref{eq:preD}, the sub-adjacent 3-Lie algebra $(A, \{\cdot,\cdot,\cdot\}_{Dc} )$ with the multiplication

$\{x,y,z\}_{Dc}=\{x,y,z\}_D+\{y,z,x\}_D+\{z,x,y\}_D$

\hspace{2cm}$=[Dx, Dy, z]+ [Dy, Dz, x]+[Dz, Dx, y]=D[Dx,Dy,Dz].$
\\ It follows Eq \eqref{eq:liepreD}.
Since for all $x, y, z\in A$,

$D(\{x,y,z\}_{Dc})=D(D[Dx, Dy, Dz]=D^2[Dx, Dy, Dz]=[Dx, Dy, Dz]$,
\\$D$ is an algebra isomorphism.
Thanks to Eqs \eqref{eq:liepreD} and \eqref{eq:23}, Eq\eqref{eq:liepreD'} holds.  \end{proof}

\begin{theorem}
Let $(A,[\cdot,\cdot,\cdot])$ be a $3$-Lie algebra with an involutive derivation $D$. Then there exists a compatible $3$-pre Lie algebra $(A, \{ , , \}_A)$, where
\begin{equation}
\{x,y,z\}_{A}=D[x,y,Dz].
\label{eq:LieP}
\end{equation}
\label{thm:Liepre}
 \end{theorem}

 \begin{proof} By  Eq \eqref{eq:LieP},
 for all $x_i\in A, 1\leq i\leq 5,$

 $\{ x_1, x_2, x_3\}_A=D[x_1, x_2, Dx_3]=-D[x_2, x_1, Dx_3]=-\{ x_2, x_1, x_3\}_A$,

\vspace{2mm}$\{\{x_{1},x_{2},\{x_{3}, x_{4},x_{5}\}_A\}_A=D[x_1, x_2, D^2[x_3, x_4, Dx_5]]=D[x_1, x_2, [x_3, x_4, Dx_5]]$,
\\
 Eq \eqref{eq:pre1} holds. And
$$[[Dx_{1},Dx_{2},x_{3}]+[Dx_{1},x_{2},Dx_{3}]+[x_{1},Dx_{2},Dx_{3}],x_{4},Dx_{5}]=-[[x_{1},x_{2},x_{3}],x_{4},Dx_{5}],$$

\vspace{2mm}$ \{\{x_{1},x_{2},x_{3}\}_{Ac},x_{4},x_{5}\}_A+\{x_{3},\{x_{1},x_{2},x_{4}\}_{Ac},x_{5}\}_A+\{x_{3},x_{4},\{x_{1},x_{2},x_{5}\}_A\}_A$

\vspace{2mm}$=D\big([D([x_{1},x_{2},Dx_{3}]+[x_{2},x_{3},Dx_{1}]+[x_{3},x_{1},Dx_{2}]),x_{4},Dx_{5}]$

\vspace{2mm}$+[x_{3},D([x_{1},x_{2},Dx_{4}]+[x_{2},x_{4},Dx_{1}]+[x_{4},x_{1},Dx_{2}]),Dx_{5}]+$$[x_{3},x_{4}, [x_{1},x_{2},Dx_{5}]]\big)$

\vspace{2mm}$
=D\big(2[[Dx_{1},x_{2},Dx_{3}],x_{4},Dx_{5}]+2[[x_{1},Dx_{2},Dx_{3}],x_{4},Dx_{5}]+2[[Dx_{1},Dx_{2},x_{3}],x_{4},Dx_{5}]$

\vspace{2mm}$+2[x_{3},[Dx_{1},Dx_{2},x_{4}],Dx_{5}]+2[x_{3},[x_{1},Dx_{2},Dx_{4}],Dx_{5}]+2[x_{3},[Dx_{1},x_{2},Dx_{4}],Dx_{5}])$

\vspace{2mm}$+3[[x_{1},x_{2},x_{3}],x_{4},Dx_{5}]+3[x_{3},[x_{1},x_{2},x_{4}],Dx_{5}]+3[x_{3},x_{4},[x_{1},x_{2},Dx_{5}]]\big)$

\vspace{2mm}$=D([x_{1},x_{2},[x_{3},x_{4},Dx_{5}]]+[[x_{1},x_{2},x_{3}],x_{4},Dx_{5}]+[x_{3},[x_{1},x_{2},x_{4}],Dx_{5}]$

\vspace{2mm}$
+[D[Dx_{1},Dx_{2},Dx_{3}],x_{4},Dx_{5}]+[x_{3},D[Dx_{1},Dx_{2},Dx_{4}],Dx_{5}])$

\vspace{2mm}$=D[x_{1},x_{2},[x_{3},x_{4},Dx_{5}]]=\{x_{1},x_{2},\{x_{3},x_{4}.x_{5}\}_A\}_A,$
\\ it follows Eq \eqref{eq:pre2}.
 Similar discussion to the above, we get Eq \eqref{eq:pre3}. Thanks to
  $$\{x_1, x_2, x_3\}_A=D([x_1, x_2, Dx_3]+[x_2, x_3, Dx_1]+[x_3, x_1, Dx_2]),$$
$(A, \{ , , \}_A)$ is the compatible 3-pre-Lie algebra of $(A, [ , , ])$.
\end{proof}

\begin{theorem}  Let $A$ be an $m$-dimensional 3-Lie algebra.
 If $m=4$, then  there are compatible $3$-pre Lie algebras.  If $m=5$, $\dim A^1\leq 3$, or, $\dim A^1=4$ and $Z(A)\neq 0$,
then there are compatible $3$-pre-Lie algebras.

\label{thm:45dim}
 \end{theorem}

\begin{proof} The result follows from the straightforward checking of the existence of involutive
derivations by the classification theorems in \cite{BSZhang} and Theorem \ref{thm:Liepre}, we omit the computation process.
\end{proof}

\section{ Local cocycle $3$-Lie bialgebras induced by involutive derivations}

In paper \cite{BGL},  the concept of 3-Lie bialgebra was introduced, but it has not close relationship with $3$-Lie classical Yang-Baxter equation.
Authors in \cite{BGS}, gave another  3-Lie bialgebraic structures on a 3-Lie algebra, it is called the  local cocycle $3$-Lie bialgebra. It has close relationship with 3-pre-Lie algebras, Manin triples, matched pairs of 3-Lie algebras and  $3$-Lie classical Yang-Baxter equation. So, in this section, we study  local cocycle $3$-Lie bialgebras induced by involutive derivations.

 \begin{defn}\cite{BGS}
A \textbf{local cocycle $3$-Lie bialgebra} is a triple $(A, [, , ], \Delta)$, where $(A, [ , , ])$ is a $3$-Lie algebra and $\Delta=\Delta_1+\Delta_2+\Delta_3:A\rightarrow A\otimes
A\otimes A$ is a linear mapping such that $(A^*, \Delta^*)$ is a 3-Lie algebra, and

\begin{itemize}
\item $\Delta_{1}$ is a $1$-cocycle associated to the representataion $(A\otimes A\otimes A,ad\otimes 1\otimes 1)$,
\item $\Delta_{2}$ is a $1$-cocycle associated to the representataion $(A\otimes A\otimes A,1\otimes ad\otimes 1)$,
\item $\Delta_{3}$ is a $1$-cocycle associated to the representataion $(A\otimes A\otimes A,1\otimes 1\otimes ad)$,
\end{itemize}
where  $\Delta^*:A^*\otimes A^*\otimes A^*\rightarrow
A^*$ is the dual mapping of $\Delta$, that is, for all $\alpha, \beta, \gamma\in A^*$ and $x\in A$,
$$ \langle \Delta^*(\alpha, \beta, \gamma), x\rangle=\langle \alpha\otimes \beta\otimes \gamma, \Delta x\rangle.$$

\label{defin:3B}
\end{defn}

Let $A$ be a 3-Lie algebra,  $r=\sum\limits_{i=1}^m x_i\otimes y_i\in A\otimes A$, $p, q\in \mathbb Z_{>0}$,  and  $1\leq p\neq q\leq m$. Define an inclusion $\cdot_{pq}:\otimes^2A\longrightarrow \otimes^n A$ by sending $r=\sum\limits_i x_i\otimes y_i\in A\otimes A$ to

\begin{equation}
r_{pq}:=\sum\limits_i z_{i1}\otimes\cdots\otimes z_{in},\quad \text{ where } z_{ij}=\left\{\begin{array}{ll} x_i,& j=p,\\ y_i, & j=q, \\ 1, & j\neq p, q,  \end{array} \right. 1\leq p, q, i, j\leq n,
\label{eq:rpq}
\end{equation}
that is, $r_{pq}$ puts $x_i$ at the $p$-th position, $y_i$
at the $q$-th position and 1 elsewhere in the $n$-tensor, where
1 is a symbol playing a similar role of unit.

Define an operator $\phi_{pq}:A^{\otimes m}\rightarrow A^{\otimes m}$,  $\forall \sum x_1\otimes x_2\otimes\cdots\otimes x_p\otimes\cdots\otimes x_q\otimes\cdots\otimes x_m\in A^{\otimes m}$,
$$
\phi_{pq}(x_1\otimes x_2\otimes\cdots\otimes x_p\otimes\cdots\otimes x_q\otimes\cdots \otimes x_m)=x_1\otimes x_2\otimes\cdots\otimes x_q\otimes\cdots\otimes x_p\otimes\otimes x_m,
$$
that is, $\phi_{pq}$ exchanges the position of $x_p$ with $x_q$, where $m\geq 2$.

\begin{defn}\cite{BGS}
Let $(A,[\cdot,\cdot,\cdot])$ be a $3$-Lie algebra and $r=\sum_i x_i\otimes y_i\in A\otimes A$, denote
\begin{equation}
\begin{split}
[[r,r,r]]: \equiv&\sum_{i,j,k}\big([x_i,x_j,x_k]\otimes y_i\otimes y_j\otimes y_k+x_i\otimes [y_i,x_j,x_k]\otimes y_j\otimes y_k\\
&+ x_i\otimes x_j\otimes [y_i, y_j,x_k]\otimes y_k+ x_i\otimes x_j\otimes x_k\otimes [y_i,y_j,y_k]\big).
\end{split}
\label{eq:rrr}
\end{equation}

The equation
\begin{equation}
[[r,r,r]]=0
\label{eq:CYBE}
\end{equation}
is called the {\bf $3$-Lie classical Yang-Baxter equation} in the 3-Lie algebra $A$, and simply denoted by {\bf CYBE}.
\end{defn}

Let $D:A\rightarrow A$ be a linear mapping. Define the tensor $\overline D\in A^{*}\otimes A$, for all $x\in A, \xi \in A^{*},$
\begin{equation}
\overline{D}(x,\xi)=\langle \xi, Dx\rangle.
\label{eq:barD}
\end{equation}

\begin{theorem}\label{thm:DAR}
Let $(A,[\cdot,\cdot,\cdot])$ be a $3$-Lie algebra with a basis $\{x_{1}, \cdots, x_{n}\}$,  $\{x^*_{1},$ $\cdots, $ $x^*_{n}\}$ be the dual basis of $A^*$,
and $D$ be an involutive derivation on $A$. Then
\begin{equation}\label{eq:Dr}
r=\overline{D}-\sigma_{12}\overline{D}
 \end{equation}
is a skew-symmetric solution  of  {\bf CYBE}  in the semi-direct product $3$-Lie algebra $A\ltimes_{ad^*} A^*$, and
\begin{equation}\label{eq:barD}
\overline{D}=\sum_{i=1}^n x_{i}^{*}\otimes Dx_i, ~~~~ r=\sum_{i=1}^n x_{i}^{*}\otimes Dx_i-\sum_{i=1}^n Dx_{i}\otimes x^*_i\in A^{*}\otimes A.
\end{equation}
\label{thm:888}
\end{theorem}

\begin{proof} By Theorem \ref{thm:DD}, $D$ is an $\wp$-operator associative to the adjoint representation $(A, ad)$.
By Theorem 3.19 in \cite{BGS}, $r=\overline{D}-\sigma_{12}\overline{D}$ is a skew-symmetric solution of  {\bf CYBE} in the semi-direct product $3$-Lie algebra $A\ltimes_{ad^*} A^*$. Thanks to $D\in Der(A)$ and $D^2=I_d$, we get Eq \eqref{eq:barD}. The proof is completed.
\end{proof}

\begin{theorem}\label{thm:666}
Let $A$ be a $3$-Lie algebra, $D$ be an involutive derivation on $A$, $\overline D$  and $r$ be defined as Eq \eqref{eq:barD} and Eq \eqref{eq:Dr}.
Then $r$ induces a local cocycle $3$-Lie bialgebra \\$(A\ltimes_{ad^*} A^*, \Delta)$,  where for all $x\in A\oplus A^*,$
\begin{equation}\label{eq:Delta1}
\begin{array}{l}
 \left\{\begin{array}{l}
\Delta(x)~=\Delta_1(x)+\Delta_2(x)+\Delta_3(x),\\
\vspace{2mm}\Delta_1(x)=\sum_{i,j=1}^n \mu(x,x_i^{*},-Dx_j)\otimes x_j^{*}\otimes Dx_i\\
\vspace{2mm}\hspace{1.5cm}+\sum_{i,j=1}^n\mu(x,-Dx_i,x_j^{*})\otimes Dx_j\otimes x_i^{*}\\
\vspace{2mm}\hspace{1.5cm}+ \sum_{i,j=1}^n\mu(x,Dx_i,Dx_j)\otimes x_j^{*}\otimes x_i^{*},\\
\vspace{2mm}\Delta_2(x)=\phi _{13}\phi _{12}\Delta_1(x),\\
\vspace{2mm}\Delta_3(x)=\phi _{12}\phi _{13}\Delta_1(x).
\end{array}\right.
\end{array}
\end{equation}

\end{theorem}

\begin{proof}
Suppose $\dim A=n$, and the decomposition of $A$ associated to $D$ is
$$ A=A_{1}+A_{-1},$$
and $\{x_1, ..., x_s, x_{s+1},...x_n\}$ is a basis of $A$, $x_1, ..., x_s\in A_1$, $x_{s+1},...x_n\in A_{-1}$.
By  Theorem {\ref{thm:DAR}}, and Eq \eqref{eq:barD}, the tensor
\begin{equation}\label{eq:r*}
 r= \sum_{i=1}^{s} (x_i^*\otimes x_i-x_i\otimes x_i^*)- \sum_{i=s+1}^{n} (x_i^*\otimes x_i-x_i\otimes x_i^*)
 \end{equation}
 is  a solution of {\bf CYBE}  in the semi-direct product 3-Lie algebra $A\ltimes_{ad^*} A^*$.

 For convenience, let
 $r=\sum_{i=1}^{2n}u_i\otimes \nu_i.$
 By \cite{BGS}, $r$ determines  a local cocycle $3$-Lie bialgebra $(A\ltimes_{ad^*} A^*, \Delta)$,
 where for all $x\in A\ltimes_{ad^*} A^*$,

 \vspace{2mm}$
\Delta_1(x)=\sum_{i,j} \mu(x,u_i,u_j)\otimes v_j\otimes v_i,$

$
\Delta_2(x)=\phi _{12}\Delta_1(x)=\sum_{i,j} v_i\otimes\mu(x,u_i,u_j)\otimes v_j,$

\vspace{2mm}
$
\Delta_3(x)=\phi _{12}\phi _{13}\Delta_1(x)=\sum_{i,j} v_j\otimes v_i\otimes \mu(x, u_i, u_j).$
Thanks to Eq \eqref{eq:r*},

\vspace{2mm}$
u_i=\left\{
\begin{array}{rcl}
x_i^*&&{1\leq i\leq n},\\
-D(x_{i-n})&&{n+1\leq i\leq 2n},
\end{array}\right.
 ~u_j=\left\{
\begin{array}{rcl}
x_j^*&&{1\leq j\leq n},\\
-D(x_{j-n})&&{n+1\leq j\leq 2n},
\end{array}\right.$

\vspace{2mm}$
\nu_i=\left\{
\begin{array}{rcl}
y_i^*&&{1\leq i\leq n},\\
D(x_{i})&&{n+1\leq i\leq 2n},
\end{array}\right.
 ~,~\nu_j=\left\{
\begin{array}{rcl}
D(x_j)&&{1\leq j\leq n},\\
x_{j-n}^*&&{n+1\leq j\leq 2n}.
\end{array}\right.$

Therefore,
\\
\vspace{2mm}$
\Delta_1(x)=\sum_{i=1}^n \sum_{j=1}^n \mu(x,u_i,u_j)\otimes \nu_j\otimes \nu_i
+\sum_{i=1}^n \sum_{j=n+1}^{2n} \mu(x,u_i,u_j)\otimes \nu_j\otimes \nu_i$

\hspace{9mm}\vspace{2mm}$+\sum_{i=n+1}^{2n}\sum_{j=1}^n \mu(x,u_i,u_j)\otimes \nu_j\otimes \nu_i+\sum_{i=n+1}^{2n} \sum_{j=n+1}^{2n} \mu(x,u_i,u_j)\otimes \nu_j\otimes \nu_i$

\hspace{5mm}\vspace{2mm}$=\sum_{i=1}^n \sum_{j=1}^n \mu(x,x_i^{*},x_j^{*})\otimes Dx_j\otimes Dx_i+\sum_{i=1}^n \sum_{j=n+1}^{2n} \mu(x,x_i^{*},-Dx_{j-n})\otimes x_{j-n}^{*}\otimes Dx_i$

\hspace{9mm}\vspace{2mm}$+\sum_{i=n+1}^{2n} \sum_{j=1}^n \mu(x,-Dx_{i-n},x_j^{*})\otimes Dx_j\otimes x_{i-n}^{*}$

\hspace{9mm}\vspace{2mm}$+\sum_{i=n+1}^{2n} \sum_{j=n+1}^{2n} \mu(x,-Dx_{i-n},-Dx_{j-n})\otimes x_{j-n}^{*}\otimes x_{i-n}^{*}
$

\hspace{5mm}\vspace{2mm}$=\sum_{i,j} \mu(x,x_i^{*},-Dx_j)\otimes x_j^{*}\otimes Dx_i
+\sum_{i,j} \mu(x,-Dx_i,x_j^{*})\otimes Dx_j\otimes x_i^{*}$

\hspace{9mm}\vspace{2mm}$+\sum_{i,j} \mu(x,Dx_i,Dx_j)\otimes x_j^{*}\otimes x_i^{*}.
$
The proof is completed.
\end{proof}

\begin{coro}
Let $(A,[\cdot,\cdot,\cdot])$ be an n-dimensional $3$-Lie algebra with an involutive derivation $D$, $ \{x_{1},\cdots,x_{s},x_{s+1},\cdots,x_{n}\}$  be a basis of $A$ and  $x_{i}\in A_{1}, 1\leq i \leq s, x_{j}\in A_{-1}, s+1\leq j\leq n.$ Then  $(A\ltimes_{ad^*} A^*, \mu, \Delta)$ is a local cocycle 3-Lie bialgebra,
where  $\Delta=\Delta_1+\Delta_2+\Delta_3$, $\Delta_2(x)=\phi _{13}\phi _{12}\Delta_1(x), \Delta_3(x)=\phi _{12}\phi _{13}\Delta_1(x),$ and

\begin{equation}\label{eq:999}
\begin{aligned}
\Delta_1(x)&=\sum_{i=1}^s \sum_{j=1}^s \mu(x,x_i^{*},-x_j)\otimes x_j^{*}\otimes x_i
+\sum_{i=1}^s \sum_{j=s+1}^n \mu(x,x_i^{*},x_j)\otimes x_j^{*}\otimes x_i\\&+\sum_{i=s+1}^n \sum_{j=1}^s \mu(x,x_i^{*},-x_j)\otimes x_j^{*}\otimes (-x_i)+\sum_{i=s+1}^n \sum_{j=s+1}^n \mu(x,x_i^{*},x_j)\otimes x_j^{*}\otimes (-x_i)\\&+\sum_{i=1}^s \sum_{j=1}^s \mu(x,-x_i,x_j^{*})\otimes x_j\otimes x_i^{*}+\sum_{i=1}^s \sum_{j=s+1}^n \mu(x,-x_i,x_j^{*})\otimes (-x_j)^{*}\otimes x_i^{*}\\&+\sum_{i=s+1}^n \sum_{j=1}^s \mu(x,x_i,x_j^{*})\otimes x_j\otimes x_i^{*}+\sum_{i=s+1}^n \sum_{j=s+1}^n \mu(x,x_i,x_j^{*})\otimes (-x_j)\otimes x_i^{*}\\&+\sum_{i=1}^s \sum_{j=1}^s \mu(x,x_i,x_j)\otimes x_j^{*}\otimes x_i^{*}+\sum_{i=1}^s \sum_{j=s+1}^n \mu(x,x_i,-x_j)\otimes x_j^{*}\otimes x_i^{*}\\&+\sum_{i=s+1}^n \sum_{j=1}^s \mu(x,-x_i,x_j)\otimes x_j^{*}\otimes x_i^{*}+\sum_{i=s+1}^n \sum_{j=s+1}^n \mu(x,x_i,x_j)\otimes x_j^{*}\otimes x_i^{*}.\\&
\end{aligned}
\end{equation}

\label{cor:666}
\end{coro}

\begin{proof} The result follows from Theorem \ref{thm:666}, direcrly.
\end{proof}

\section{ Eight and Ten dimensional 3-Lie bialgebras}

In this section, we construct 8 and 10-dimensional  local cocycle $3$-Lie bialgebras by involutive derivations. We need the  classification theorems of 4 and 5-dimensional  3-Lie algebras in \cite{BSZhang}.

\begin{lemma} \cite{BSZhang} \label{lem:4}
Let $(A, [ , , ])$ be a  4-dimensional non-abelian $3$-Lie algebra with a basis $\{x_1,x_2,x_3,x_4\}$. Then up to isomorphisms, $A$ is  one and only one of the following possibilities

$(b_1)  ~ [x_{2}, x_{3}, x_{4}]= x_{1}$;
$(b_{2}) ~ [x_{1}, x_{2}, x_{3}]=x_{1};$ $ (c_{3}) \begin{array}{ll}
  \left\{\begin{array}{l}
{[}x_{1}, x_{3}, x_{4}] = x_{1},\\
{[}x_{2}, x_{3}, x_{4}]= x_{2};
\end{array}\right.
\end{array}$

$(c_{1})
\begin{array}{l}
 \left\{\begin{array}{l}
{[}x_{2}, x_{3}, x_{4}] = x_{1},\\
{[}x_{1}, x_{3}, x_{4}] = x_{2};
\end{array}\right.
~(c_{2}) \left\{\begin{array}{l}
{[} x_{2}, x_{3}, x_{4}] =\alpha x_{1}+ x_{2},\\
{[}x_{1}, x_{3}, x_{4}] = e_{2}, \alpha\in F$ and $\alpha \neq 0;
\end{array}\right.
\end{array}$

~$ (d_{1}) \begin{array}{lll}
  \left\{\begin{array}{l}
{[}x_{2}, x_{3}, x_{4}] = x_{1},\\
{[}x_{1}, x_{3}, x_{4}] = x_{2},\\
{[}x_{1}, x_{2}, x_{3}] = x_{3};
\end{array}\right.
\end{array}$
$ (e_{1}) \begin{array}{llll}
  \left\{\begin{array}{l}
{[}x_{2}, x_{3}, x_{4}] = -x_{2},\\
{[}x_{1}, x_{3}, x_{4}] = x_{1},\\
{[}x_{1}, x_{2}, x_{3}]= x_{3},\\
{[}x_{1}, x_{2}, x_{4}] = -x_{4}.
\end{array}\right.
\end{array}$
\label{lem:5}
\end{lemma}

\begin{lemma} \cite{BSZhang}
Let $(A,[\cdot, \cdot,\cdot])$ be a  5-dimensional  non-abelian $3$-Lie algebra with a basis $\{x_1,x_2,x_3,$ $x_4,$ $x_5\}$.  Then up to isomorphisms, $A$ is one and only one of the following possibilities

$(b_1)  ~ {[}x_{2}, x_{3}, x_{4}] = x_{1}$;
$(b_{2}) ~ [x_{1}, x_{2}, x_{3}]=x_{1},$

$(c_{1})
\begin{array}{l}
 \left\{\begin{array}{l}
{[} x_{2}, x_{3}, x_{4}] = x_{1},\\
{[}x_{3}, x_{4}, x_{5}] = x_{2};
\end{array}\right.
~(c_{2}) \left\{\begin{array}{l}
{[} x_{2}, x_{3}, x_{4}] = x_{1},\\
{[}x_{2}, x_{4}, x_{5}] = x_{2}, \\
{[}x_{1}, x_{4}, x_{5}] = x_{1};
\end{array}\right.
(c_{3}) \left\{\begin{array}{l}
{[}x_{2}, x_{3}, x_{4}] = x_{1},\\
{[}x_{1}, x_{3}, x_{4}] = x_{2};
\end{array}\right.
\end{array}$

$(c_{4})
\begin{array}{ll}
 \left\{\begin{array}{l}
{[}x_{2}, x_{3}, x_{4}] = x_{1},\\
{[}x_{1}, x_{3}, x_{4}] = x_{2},\\
{[}x_{2}, x_{4}, x_{5}] = x_{2},\\
{[}x_{1}, x_{4}, x_{5}] = x_{1};\\
\end{array}\right.
~(c_{5}) \left\{\begin{array}{l}
{[} x_{2}, x_{3}, x_{4}] = x_{1},\\
{[}x_{2}, x_{4}, x_{5}] = \alpha x_{1}+x_{2};
\end{array}\right.
\end{array}$

$ (c_{6}) \begin{array}{ll}
  \left\{\begin{array}{l}
{[}x_{2}, x_{3}, x_{4}] = \alpha x_{1}+x_{2},\\
{[}x_{1}, x_{3}, x_{4}] = x_{2},\\
{[}x_{2}, x_{4}, x_{5}] = x_{2},\\
{[}x_{1}, x_{4}, x_{5}] = x_{1};
\end{array}\right.
(c_{7}) \begin{array}{ll}
  \left\{\begin{array}{l}
{[}x_{1}, x_{3}, x_{4}] = x_{1},\\
{[}x_{2}, x_{3}, x_{4}] = x_{2};
\end{array}\right.
\end{array}
\end{array}$
$\alpha\in F$ and $\alpha \neq 0.$

$ (d_{1}) \begin{array}{lll}
  \left\{\begin{array}{l}
{[}x_{2}, x_{3}, x_{4}] = x_{1},\\
{[}x_{2}, x_{4}, x_{5}] = -x_{2},\\
{[}x_{3}, x_{4}, x_{5}] = x_{3};
\end{array}\right.
(d_{2})  \left\{\begin{array}{l}
{[} x_{2}, x_{3}, x_{4}] = x_{1},\\
{[}x_{3}, x_{4}, x_{5}] = \alpha x_{2}+x_{3};
\end{array}\right.
\end{array}$

$(d_{2})
\begin{array}{l}
 \left\{\begin{array}{l}
{[} x_{2}, x_{3}, x_{4}] = x_{1},\\
{[}x_{3}, x_{4}, x_{5}] = \alpha x_{2}+x_{3};
\end{array}\right.
~(d_{3}) \left\{\begin{array}{l}
{[} x_{2}, x_{3}, x_{4}] = x_{1},\\
{[}x_{3}, x_{4}, x_{5}] = x_{3}, \\
{[}x_{2}, x_{4}, x_{5}] = x_{2},\\
{[}x_{1}, x_{4}, x_{5}] = 2x_{1};\\
\end{array}\right.
\end{array}$

$(d_{4}) \begin{array}{l}
 \left\{\begin{array}{l}
{[} x_{2}, x_{3}, x_{4}] = x_{1},\\
{[}x_{1}, x_{3}, x_{4}] = x_{2},\\
{[}x_{1}, x_{2}, x_{4}] = x_{3};
\end{array}\right.
~(d_{5}) \left\{\begin{array}{l}
{[} x_{1}, x_{4}, x_{5}] = x_{1},\\
{[}x_{2}, x_{4}, x_{5}] = x_{3}, \\
{[}x_{3}, x_{4}, x_{5}] = \beta x_{2}+(1+\beta)x_{3};
\end{array}\right.
\end{array}$

 $(d_{6})
\begin{array}{l}
 \left\{\begin{array}{l}
{[} x_{1}, x_{4}, x_{5}] = x_{1},\\
{[}x_{2}, x_{4}, x_{5}] = x_{2},\\
{[}x_{3}, x_{4}, x_{5}] = x_{3};
\end{array}\right.
(d_{7}) \left\{\begin{array}{l}
{[} x_{1}, x_{4}, x_{5}] = x_{2},\\
{[}x_{2}, x_{4}, x_{5}] = x_{3}, \\
{[}x_{3}, x_{4}, x_{5}] =sx_{1}+tx_{2}+ux_{3};
\end{array}\right.
\end{array}$
\\$\beta,s,t,u\in F$, $\beta s\neq 0$.

$(e_{1})
\begin{array}{l}
 \left\{\begin{array}{l}
{[} x_{2}, x_{3}, x_{4}] = x_{1},\\
{[}x_{3}, x_{4}, x_{5}] = x_{2},\\
{[}x_{2}, x_{4}, x_{5}] = x_{3},\\
{[}x_{2}, x_{3}, x_{5}] = x_{4};
\end{array}\right.
~(e_{2}) \left\{\begin{array}{l}
{[} x_{2}, x_{3}, x_{4}] = x_{1},\\
{[}x_{1}, x_{3}, x_{4}] = x_{2}, \\
{[} x_{1}, x_{2}, x_{4}] = x_{3},\\
{[}x_{1}, x_{2}, x_{3}] = x_{4}.
\end{array}\right.
\end{array}$

\label{lem:5}
\end{lemma}

\begin{theorem}\label{thm:41}
Let $(A,[\cdot, \cdot,\cdot])$ be a 4-dimensional $3$-Lie algebra with a basis $\{x_1,x_2,$ $x_3,x_4\}$.
Then we have  $8$-dimensional local cocycle $3$-Lie bialgebras  $( A\ltimes_{ad^*} A^*,\mu_{i}, \Delta^{i})$, $1\leq i\leq 7$, where
    \begin{equation}
\begin{array}{llll}
   \left\{\begin{array}{l}
\vspace{2mm}\mu_{1}(x_{2},x_{3},x_{4})=x_{1},\\
\vspace{2mm}\mu_{1}(x_{2},x_{3},x_{1}^{*})=-x_{4}^{*},\\
\vspace{2mm}\mu_{1}(x_{2},x_{4},x_{1}^{*})=x_{3}^{*},\\
\vspace{2mm}\mu_{1}(x_{3},x_{4},x_{1}^{*})=-x_{2}^{*}.
\end{array}\right.
\end{array}
~~ \begin{array}{llllllll}
   \left\{\begin{array}{l}
\vspace{2mm}\Delta^{1}(x_{1}^{*})=x_{2}^{*}\wedge x_{4}^{*}\wedge x_{3}^{*},\\
\vspace{2mm}\Delta^{1}(x_{2})=x_{1}\wedge x_{3}^{*}\wedge x_{4}^{*},\\
\vspace{2mm}\Delta^{1}(x_{3})=x_{1}\wedge x_{4}^{*}\wedge x_{2}^{*},\\
\vspace{2mm}\Delta^{1}(x_{4})=x_{1}\wedge x_{2}^{*}\wedge x_{3}^{*},\\
\vspace{2mm}\Delta^{1}(x_{2}^{*})=\Delta^{1}(x_{3}^{*})=\Delta^{1}(x_{4}^{*})=\Delta^{1}(x_{1})=0.
\end{array}\right.
\end{array}
\label{eq:4b1}
\end{equation}

 \begin{equation}\label{eq:4b2}
\begin{array}{llll}
   \left\{\begin{array}{l}
\vspace{2mm}\mu_{2}(x_{1},x_{2},x_{3})=x_{1},\\
\vspace{2mm}\mu_{2}(x_{1},x_{2},x_{1}^{*})=-x_{3}^{*},\\
\vspace{2mm}\mu_{2}(x_{2},x_{3},x_{1}^{*})=-x_{1}^{*},\\
\vspace{2mm}\mu_{2}(x_{1},x_{3},x_{1}^{*})=x_{2}^{*}.
\end{array}\right.
\end{array}
\begin{array}{llllllll}
   \left\{\begin{array}{l}
 \vspace{2mm}\Delta^{2}(x_{1}^{*})=x_{1}^{*}\wedge x_{3}^{*}\wedge x_{2}^{*},\\
 \vspace{2mm}\Delta^{2}(x_{1})=x_{1}\wedge x_{2}^{*}\wedge x_{3}^{*},\\
 \vspace{2mm}\Delta^{2}(x_{2})=x_{1}\wedge x_{3}^{*}\wedge x_{1},\\
 \vspace{2mm}\Delta^{2}(x_{3})=x_{1}\wedge x_{1}^{*}\wedge x_{2}^{*},\\
 \vspace{2mm}\Delta^{2}(x_{4})=\Delta^{2}(x_{2}^{*})=\Delta^{2}(x_{3}^{*})=\Delta^{2}(x_{4}^{*})=0.
\end{array}\right.
\end{array}
\end{equation}
\begin{equation}\label{eq:4c1}
\begin{array}{llll}
   \left\{\begin{array}{l}
\vspace{2mm}\mu_{3}(x_{2},x_{3},x_{4})=x_{1},\\
\vspace{2mm}\mu_{3}(x_{1},x_{3},x_{4})=x_{2},\\
\vspace{2mm}\mu_{3}(x_{2},x_{3},x_{1}^{*})=-x_{4}^{*},\\
\vspace{2mm}\mu_{3}(x_{3},x_{4},x_{1}^{*})=-x_{2}^{*},\\
\vspace{2mm}\mu_{3}(x_{2},x_{4},x_{1}^{*})=x_{3}^{*},\\
\vspace{2mm}\mu_{3}(x_{1},x_{3},x_{2}^{*})=-x_{4}^{*},\\
\vspace{2mm}\mu_{3}(x_{1},x_{4},x_{2}^{*})=x_{3}^{*},\\
\vspace{2mm}\mu_{3}(x_{3},x_{4},x_{2}^{*})=-x_{1}^{*}.
\end{array}\right.
\end{array}
\begin{array}{llllllll}
   \left\{\begin{array}{l}
\vspace{2mm}\Delta^{3}(x_{1}^{*})=x_{2}^{*}\wedge x_{4}^{*}\wedge x_{3}^{*},\\
\vspace{2mm}\Delta^{3}(x_{2}^{*})=x_{1}^{*}\wedge x_{4}^{*}\wedge x_{3}^{*},\\
\vspace{2mm}\Delta^{3}(x_{1})=x_{2}\wedge x_{3}^{*}\wedge x_{4}^{*},\\
\vspace{2mm}\Delta^{3}(x_{2})=x_{1}\wedge x_{3}^{*}\wedge x_{4}^{*},\\
\vspace{2mm}\Delta^{3}(x_{3})=x_{1}\wedge x_{4}^{*}\wedge x_{2}^{*}+x_{2}\wedge x_{4}^{*}\wedge x_{1}^{*},\\
\vspace{2mm}\Delta^{3}(x_{4})=x_{1}\wedge x_{2}^{*}\wedge x_{3}^{*}+x_{1}^{*}\wedge x_{3}^{*}\wedge x_{2},\\
\vspace{2mm}\Delta^{3}(x_{3}^{*})=\Delta^{3}(x_{4}^{*})=0.
\end{array}\right.
\end{array}
\end{equation}

\begin{equation}\label{eq:4c2}
\begin{array}{llllllllll}
   \left\{\begin{array}{l}
\vspace{2mm}\mu_{4}(x_{2},x_{3},x_{4})=\alpha x_{1}+x_{2},\\
\vspace{2mm}\mu_{4}(x_{1},x_{3},x_{4})=x_{2},\\
\vspace{2mm}\mu_{4}(x_{2},x_{3},x_{1}^{*})=-\alpha x_{4}^{*},\\
\vspace{2mm}\mu_{4}(x_{2},x_{3},x_{2}^{*})=-x_{4}^{*},\\
\vspace{2mm}\mu_{4}(x_{2},x_{4},x_{1}^{*})=\alpha x_{3}^{*},\\
\vspace{2mm}\mu_{4}(x_{2},x_{4},x_{2}^{*})=x_{3}^{*},\\
\vspace{2mm}\mu_{4}(x_{3},x_{4},x_{1}^{*})=-\alpha x_{2}^{*},\\
\vspace{2mm}\mu_{4}(x_{3},x_{4},x_{2}^{*})=-x_{1}^{*}-x_{2}^{*},\\
\vspace{2mm}\mu_{4}(x_{1},x_{3},x_{2}^{*})=-x_{4}^{*},\\
\vspace{2mm}\mu_{4}(x_{1},x_{4},x_{2}^{*})=x_{3}^{*}.
\end{array}\right.
\end{array}
\begin{array}{lllllll}
\left\{\begin{array}{l}
\vspace{2mm}\Delta^{4}(x_{1}^{*})=\alpha x_{3}^{*}\wedge x_{2}^{*}\wedge x_{4}^{*}\\
\vspace{2mm}\Delta^{4}(x_{2}^{*})=x_{1}^{*}\wedge x_{4}^{*}\wedge x_{3}^{*}+ x_{2}^{*}\wedge x_{4}^{*}\wedge x_{3}^{*},\\
\vspace{2mm}\Delta^{4}(x_{1})=x_{2}\wedge x_{3}^{*}\wedge x_{4}^{*},\\
\vspace{2mm}\Delta^{4}(x_{2})= \alpha x_{1}\wedge x_{3}^{*}\wedge x_{4}^{*}+x_{3}^{*}\wedge x_{4}^{*}\wedge x_{2},\\
\vspace{2mm}\Delta^{4}(x_{3})=\alpha x_{1}\wedge x_{4}^{*}\wedge x_{2}^{*}+x_{2}\wedge x_{4}^{*}\wedge x_{2}^{*}\\
\hspace{1.5cm}+x_{1}^{*}\wedge x_{2}\wedge x_{4}^{*},\\
\vspace{2mm}\Delta^{4}(x_{4})=\alpha x_{1}\wedge x_{2}^{*}\wedge x_{3}^{*}+x_{2}\wedge x_{2}^{*}\wedge x_{3}^{*}\\
\hspace{1.5cm}+x_{1}^{*}\wedge x_{3}^{*}\wedge x_{2},\\
\vspace{2mm}\Delta^{4}(x_{3}^{*})=\Delta^{4}(x_{4}^{*})=0.
\end{array}\right.
\end{array}
\end{equation}
\begin{equation}
\begin{array}{llll}
   \left\{\begin{array}{l}
\vspace{2mm}\mu_{6}(x_{2},x_{3},x_{4})=x_{1},\\
\vspace{2mm}\mu_{6}(x_{1},x_{3},x_{4})=x_{2},\\
\vspace{2mm}\mu_{6}(x_{1},x_{2},x_{4})=x_{3},\\
\vspace{2mm}\mu_{6}(x_{2},x_{3},x_{1}^{*})=-x_{4}^{*},\\
\vspace{2mm}\mu_{6}(x_{3},x_{4},x_{1}^{*})=-x_{2}^{*},\\
\vspace{2mm}\mu_{6}(x_{2},x_{4},x_{1}^{*})=x_{3}^{*},\\
\vspace{2mm}\mu_{6}(x_{1},x_{3},x_{2}^{*})=-x_{4}^{*},\\
\vspace{2mm}\mu_{6}(x_{1},x_{4},x_{2}^{*})=x_{3}^{*},\\
\vspace{2mm}\mu_{6}(x_{3},x_{4},x_{2}^{*})=-x_{1}^{*},\\
\vspace{2mm}\mu_{6}(x_{1},x_{2},x_{3}^{*})=-x_{4}^{*},\\
\vspace{2mm}\mu_{6}(x_{1},x_{4},x_{3}^{*})=x_{2}^{*},\\
\vspace{2mm}\mu_{6}(x_{2},x_{4},x_{3}^{*})=-x_{1}^{*}.
\end{array}\right.
\end{array}\begin{array}{llllllll}
   \left\{\begin{array}{l}
\vspace{2mm}\Delta^{6}(x_{1}^{*})=x_{2}^{*}\wedge x_{4}^{*}\wedge x_{3}^{*},\\
\vspace{2mm}\Delta^{6}(x_{2}^{*})=x_{1}^{*}\wedge x_{4}^{*}\wedge x_{3}^{*},\\
\vspace{2mm}\Delta^{6}(x_{3}^{*})=x_{1}^{*}\wedge x_{4}^{*}\wedge x_{2}^{*}\\
\vspace{2mm}\Delta^{6}(x_{1})=x_{2}\wedge x_{3}^{*}\wedge x_{4}^{*}+x_{2}^{*}\wedge x_{4}^{*}\wedge x_{3},\\
\vspace{2mm}\Delta^{6}(x_{2})=x_{1}\wedge x_{3}^{*}\wedge x_{4}^{*}+x_{4}^{*}\wedge x_{1}^{*}\wedge x_{3},\\
\vspace{2mm}\Delta^{6}(x_{3})=x_{4}^{*}\wedge x_{2}^{*}\wedge x_{1}+x_{4}^{*}\wedge x_{1}^{*}\wedge x_{2},\\
\vspace{2mm}\Delta^{6}(x_{4})=x_{1}^{*}\wedge x_{3}^{*}\wedge x_{2}+x_{1}^{*}\wedge x_{2}^{*}\wedge x_{3},\\
\vspace{2mm}\hspace{1.5cm}+x_{1}\wedge x_{2}^{*}\wedge x_{3}^{*}, \\
\vspace{2mm}\Delta^{6}(x_{4}^{*})=0.
\end{array}\right.
\end{array}
\end{equation}
\begin{equation}
\begin{array}{llll}
   \left\{\begin{array}{l}
\vspace{2mm}\mu_{7}(x_{2},x_{3},x_{4})=-x_{2},\\
\vspace{2mm}\mu_{7}(x_{1},x_{3},x_{4})=x_{1},\\
\vspace{2mm}\mu_{7}(x_{1},x_{2},x_{3})=x_{3},\\
\vspace{2mm}\mu_{7}(x_{1},x_{2},x_{4})=-x_{4},\\
\vspace{2mm}\mu_{7}(x_{2},x_{3},x_{2}^{*})=x_{4}^{*},\\
\vspace{2mm}\mu_{7}(x_{3},x_{4},x_{2}^{*})=x_{2}^{*},\\
\vspace{2mm}\mu_{7}(x_{2},x_{4},x_{2}^{*})=-x_{3}^{*},\\
\vspace{2mm}\mu_{7}(x_{1},x_{3},x_{1}^{*})=-x_{4}^{*},\\
\vspace{2mm}\mu_{7}(x_{3},x_{4},x_{1}^{*})=-x_{1}^{*},\\
\vspace{2mm}\mu_{7}(x_{1},x_{4},x_{1}^{*})=x_{3}^{*},\\
\vspace{2mm}\mu_{7}(x_{1},x_{2},x_{3}^{*})=-x_{3}^{*},\\
\vspace{2mm}\mu_{7}(x_{1},x_{3},x_{3}^{*})=x_{2}^{*},\\
\vspace{2mm}\mu_{7}(x_{2},x_{3},x_{3}^{*})=-x_{1}^{*},\\
\vspace{2mm}\mu_{7}(x_{1},x_{2},x_{4}^{*})=x_{4}^{*},\\
\vspace{2mm}\mu_{7}(x_{1},x_{4},x_{4}^{*})=-x_{2}^{*},\\
\vspace{2mm}\mu_{7}(x_{2},x_{4},x_{4}^{*})=x_{1}^{*}.
\end{array}\right.
\end{array}
\begin{array}{llllllll}
   \left\{\begin{array}{l}
\vspace{2mm}\Delta^{7}(x_{1}^{*})=x_{1}^{*}\wedge x_{4}^{*}\wedge x_{3}^{*},\\
\vspace{2mm}\Delta^{7}(x_{2}^{*})=x_{2}^{*}\wedge x_{3}^{*}\wedge x_{4}^{*},\\
\vspace{2mm}\Delta^{7}(x_{3}^{*})=x_{1}^{*}\wedge x_{3}^{*}\wedge x_{2}^{*},\\
\vspace{2mm}\Delta^{7}(x_{4}^{*})=x_{1}^{*}\wedge x_{2}^{*}\wedge x_{4}^{*},\\
\vspace{2mm}\Delta^{7}(x_{1})=x_{1}\wedge x_{3}^{*}\wedge x_{4}^{*}+x_{2}^{*}\wedge x_{3}^{*}\wedge x_{3}\\
\vspace{2mm}\hspace{1.5cm}+x_{4}^{*}\wedge x_{2}^{*}\wedge x_{4},\\
\vspace{2mm}\Delta^{7}(x_{2})=x_{2}\wedge x_{4}^{*}\wedge x_{3}^{*}+x_{3}^{*}\wedge x_{1}^{*}\wedge x_{3}\\
\vspace{2mm}\hspace{1.5cm}+x_{1}^{*}\wedge x_{4}^{*}\wedge x_{4},\\
\vspace{2mm}\Delta^{7}(x_{3})=x_{4}^{*}\wedge x_{1}^{*}\wedge x_{1}+x_{2}^{*}\wedge x_{4}^{*}\wedge x_{2}\\
\vspace{2mm}\hspace{1.5cm}+x_{1}^{*}\wedge x_{2}^{*}\wedge x_{3},\\
\vspace{2mm}\Delta^{7}(x_{4})=x_{3}^{*}\wedge x_{2}^{*}\wedge x_{2}+x_{1}^{*}\wedge x_{3}^{*}\wedge x_{1}\\
\vspace{2mm}\hspace{1.5cm}+x_{4}\wedge x_{2}^{*}\wedge x_{1}^{*},\\
\vspace{2mm}\Delta^{7}(x_{4}^{*})=0.
\end{array}\right.
\end{array}
\end{equation}
\begin{equation}
\begin{array}{llll}\label{eq:4c3}
   \left\{\begin{array}{l}
\vspace{2mm}\mu_{5}(x_{1},x_{3},x_{4})=x_{1},\\
\vspace{2mm}\mu_{5}(x_{2},x_{3},x_{4})=x_{2},\\
\vspace{2mm}\mu_{5}(x_{1},x_{3},x_{1}^{*})=-x_{4}^{*},\\
\vspace{2mm}\mu_{5}(x_{1},x_{4},x_{1}^{*})=x_{3}^{*},\\
\vspace{2mm}\mu_{5}(x_{3},x_{4},x_{1}^{*})=-x_{1}^{*},\\
\vspace{2mm}\mu_{5}(x_{2},x_{3},x_{2}^{*})=-x_{4}^{*},\\
\vspace{2mm}\mu_{5}(x_{2},x_{4},x_{2}^{*})=x_{3}^{*},\\
\vspace{2mm}\mu_{5}(x_{3},x_{4},x_{2}^{*})=-x_{2}^{*}.
\end{array}\right.
\end{array}
\begin{array}{llllllll}
   \left\{\begin{array}{l}
\vspace{2mm}\Delta^{5}(x_{1}^{*})=x_{1}^{*}\wedge x_{4}^{*}\wedge x_{3}^{*},\\
\vspace{2mm}\Delta^{5}(x_{2}^{*})=x_{2}^{*}\wedge x_{4}^{*}\wedge x_{3}^{*},\\
\vspace{2mm}\Delta^{5}(x_{1})=x_{1}\wedge x_{3}^{*}\wedge x_{4}^{*},\\
\vspace{2mm}\Delta^{5}(x_{2})=x_{2}\wedge x_{3}^{*}\wedge x_{4}^{*},\\
\vspace{2mm}\Delta^{5}(x_{3})=x_{1}\wedge x_{4}^{*}\wedge x_{1}^{*}+x_{2}\wedge x_{4}^{*}\wedge x_{2}^{*},\\
\vspace{2mm}\Delta^{5}(x_{4})=x_{1}\wedge x_{1}^{*}\wedge x_{3}^{*}+x_{2}\wedge x_{2}^{*}\wedge x_{3}^{*},\\
\vspace{2mm}\Delta^{5}(x_{3}^{*})=\Delta^{5}(x_{4}^{*})=0.
\end{array}\right.
\end{array}
\end{equation}
\end{theorem}

\begin{proof} By Lemma \ref{lem:4}, and Eq \eqref{eq:semiprod}, if $A$ is the case of $(b_1)$, then there is an involutive derivation $D$ on $A$ such that $x_1, x_2, x_3\in A_1$ and $x_4\in A_{-1}$. Then
 the multiplication of the semi-direct product 3-Lie algebra $( A\ltimes_{ad^*} A^*,\mu_{1})$ in  the basis $\{x_1, \cdots, x_4$, $x_1^*, \cdots, x_4^*\}$
is as follows

$\mu_{1}(x_{2},x_{3},x_{4})=x_{1},$ ~~$\mu_{1}(x_{2},x_{3},x_{1}^{*})=ad^{*}(x_{2},x_{3})(x_{1}^{*})=-x_{4}^{*},$

$\mu_{1}(x_{2},x_{4},x_{1}^{*})=ad^{*}(x_{2},x_{4})(x_{1}^{*})=-x_{3}^{*},$
~~$\mu_{1}(x_{3},x_{4},x_{1}^{*})=ad^{*}(x_{3},x_{4})(x_{1}^{*})=-x_{2}^{*}.
$
\\
Thanks to Theorem \ref{thm:888} and Theorem \ref{thm:666},  $( A\ltimes_{ad^*} A^*,\mu_{1}, \Delta^{1})$ is  a local cocycle $3$-Lie bialgebra, where

$
\Delta_{1}^{1}(x_{1}^{*})=[x_{1}^{*},x_2,x_3]\otimes x_3^{*}\otimes x_2^{*}
+[x_{1}^{*},x_3,x_2]\otimes x_2^{*}\otimes x_3^{*}+[x_{1}^{*},x_2,-x_4]\otimes x_4^{*}\otimes x_2^{*}$

\hspace{15mm}$+[x_{1}^{*},x_3,-x_4]\otimes x_4^{*}\otimes x_3^{*}+[x_{1}^{*},-x_4,x_2]\otimes x_2^{*}\otimes x_4^{*}
+[x_{1}^{*},-x_4,x_3]\otimes x_3^{*}\otimes x_4^{*}$

\hspace{1.4cm}$=-x_{4}^{*}\otimes x_3^{*}\otimes x_2^{*}+x_{4}^{*}\otimes x_2^{*}\otimes x_3^{*}
-x_{3}^{*}\otimes x_4^{*}\otimes x_2^{*}$

\hspace{1.5cm}
$+x_2^{*}\otimes x_4^{*}\otimes x_3^{*}+x_{3}^{*}\otimes x_2^{*}\otimes x_4^{*}
-x_{2}^{*}\otimes x_3^{*}\otimes x_4^{*},$

$\Delta_{2}^{1}(x_{1}^{*})=\phi _{13}\phi_{12}\Delta_{1}^{1}(x_{1}^{*}),~~~ \Delta_{3}^{1}(x_{1}^{*})=\phi _{12}\phi_{13}\Delta_{1}^{1}(x_{1}^{*}),$

$\Delta^{1}(x_{1}^{*})=\Delta_{1}^{1}(x_{1}^{*})+\Delta_{2}^{1}(x_{1}^{*})+\Delta_{3}^{1}(x_{1}^{*})$

\hspace{15mm}$
=-x_{4}^{*}\otimes x_3^{*}\otimes x_2^{*}-x_{4}^{*}\otimes x_2^{*}\otimes x_3^{*}
+x_{2}^{*}\otimes x_4^{*}\otimes x_3^{*}$

\hspace{15mm}$
-x_2^{*}\otimes x_3^{*}\otimes x_4^{*}+x_{3}^{*}\otimes x_2^{*}\otimes x_4^{*}
-x_{3}^{*}\otimes x_4^{*}\otimes x_2^{*}=x_2^{*}\wedge x_4^{*} \wedge x_3^{*}.
$

Similarly,  we have  $\Delta^{1}(x_{2})=x_{1}\wedge x_{3}^{*}\wedge x_{4}^{*},$ $\Delta^{1}(x_{3})=x_{1}\wedge x_{4}^{*}\wedge x_{2}^{*},$
$\Delta^{1}(x_{4})=x_{1}\wedge x_{2}^{*}\wedge x_{3}^{*},$ and $\Delta^{1}(x_{2}^{*})=\Delta^{1}(x_{3}^{*})=\Delta^{1}(x_{4}^{*})=\Delta^{1}(x_{1})=0.$
 It follows Eq \eqref{eq:4b1}.
  Eq \eqref{eq:4b2} follows from the completely similar discussion.

For the case $\dim A^1=2,$ we get local cocycle $3$-Lie bialgebras $( A\ltimes_{ad^*} A^*, \mu_{i}, \Delta^{i})$), $i=3, 4, 5$. If $\dim A^1=3,$ then we get local cocycle $3$-Lie bialgebras $( A\ltimes_{ad^*} A^*, \mu_{i}, \Delta^{i})$), $i=5, 6$.
If $\dim A^1=4,$ then we get local cocycle $3$-Lie bialgebras $( A\ltimes_{ad^*} A^*, \mu_{7}, \Delta^{7})$).  We omit the proving process since it is similar to the case $\dim A^1=1$.
 \end{proof}

\begin{theorem}\label{thm:5} Let $A$ be  a 5-dimensional 3-Lie algebra with a basis $\{x_1,x_2,$ $x_3,x_4,$ $x_5\}$ and satisfy  $\dim A^1\leq 3$, or, $\dim A^1=4$ and $Z(A)\neq 0$. Then we have  $10$-dimensional local cocycle $3$-Lie bialgebras  $( A\ltimes_{ad^*} A^*,\psi_{j}, \widetilde{\Delta}^{j})$, $1\leq j\leq 17$, where $ (A\ltimes_{ad^*} A^*, \psi_{j})$ are semi-direct product  3-Lie algebras, and
\begin{equation}
\begin{array}{llll}
   \left\{\begin{array}{l}
\vspace{2mm}\psi_{1}(x_{2},x_{3},x_{4})=x_{1},\\
\vspace{2mm}\psi_{1}(x_{2},x_{3},x_{1}^{*})=-x_{4}^{*},\\
\vspace{2mm}\psi_{1}(x_{3},x_{4},x_{1}^{*})=-x_{2}^{*},\\
\vspace{2mm}\psi_{1}(x_{2},x_{4},x_{1}^{*})=x_{3}^{*}.
\end{array}\right.
\end{array}
\begin{array}{llllllllll}
   \left\{\begin{array}{l}
\vspace{2mm}\widetilde{\Delta}^{1}(x_{1}^{*})=x_{2}^{*}\wedge x_{4}^{*}\wedge x_{3}^{*},\\
\vspace{2mm}\widetilde{\Delta}^{1}(x_{2})=x_{1}\wedge x_{3}^{*}\wedge x_{4}^{*},\\
\vspace{2mm}\widetilde{\Delta}^{1}(x_{3})=x_{1}\wedge x_{4}^{*}\wedge x_{2}^{*},\\
\vspace{2mm}\widetilde{\Delta}^{1}(x_{4})=x_{1}\wedge x_{2}^{*}\wedge x_{3}^{*},\\
\vspace{2mm}\widetilde{\Delta}^{1}(x_{2}^{*})=\widetilde{\Delta}^{1}(x_{5})=\widetilde{\Delta}^{1}(x_{3}^{*})=0,\\
\vspace{2mm}\widetilde{\Delta}^{1}(x_{4}^{*})=\widetilde{\Delta}^{1}(x_{5}^{*})=\widetilde{\Delta}^{1}(x_{1})=0.
\end{array}\right.
\end{array}
\end{equation}
\begin{equation}
\begin{array}{llll}
   \left\{\begin{array}{l}
\vspace{2mm}\psi_{2}(x_{2},x_{3},x_{4})=x_{1},\\
\vspace{2mm}\psi_{2}(x_{3},x_{4},x_{5})=x_{2},\\
\vspace{2mm}\psi_{2}(x_{2},x_{3},x_{1}^{*})=-x_{4}^{*},\\
\vspace{2mm}\psi_{2}(x_{2},x_{4},x_{1}^{*})=x_{3}^{*},\\
\vspace{2mm}\psi_{2}(x_{3},x_{4},x_{1}^{*})=-x_{2}^{*},\\
\vspace{2mm}\psi_{2}(x_{3},x_{4},x_{2}^{*})=-x_{5}^{*},\\
\vspace{2mm}\psi_{2}(x_{4},x_{5},x_{2}^{*})=-x_{3}^{*},\\
\vspace{2mm}\psi_{2}(x_{3},x_{5},x_{2}^{*})=x_{4}^{*}.
\end{array}\right.
\end{array}
\begin{array}{llllllllll}
   \left\{\begin{array}{l}
\vspace{2mm}\widetilde{\Delta}^{2}(x_{1}^{*})=x_{2}^{*}\wedge x_{4}^{*}\wedge x_{3}^{*},\\
\vspace{2mm}\widetilde{\Delta}^{2}(x_{2}^{*})=x_{3}^{*}\wedge x_{5}^{*}\wedge x_{4}^{*},\\
\vspace{2mm}\widetilde{\Delta}^{2}(x_{2})=x_{1}\wedge x_{3}^{*}\wedge x_{4}^{*},\\
\vspace{2mm}\widetilde{\Delta}^{2}(x_{3})=x_{1}\wedge x_{4}^{*}\wedge x_{2}^{*}+x_{2}\wedge x_{4}^{*}\wedge x_{5}^{*},\\
\vspace{2mm}\widetilde{\Delta}^{2}(x_{4})=x_{1}\wedge x_{2}^{*}\wedge x_{3}^{*}+x_{5}^{*}\wedge x_{3}^{*}\wedge x_{2},\\
\vspace{2mm}\widetilde{\Delta}^{2}(x_{5})=x_{2}\wedge x_{3}^{*}\wedge x_{4}^{*},\\
\vspace{2mm}\widetilde{\Delta}^{2}(x_{3}^{*})=\widetilde{\Delta}^{2}(x_{4}^{*})=\widetilde{\Delta}^{2}(x_{5}^{*})=\widetilde{\Delta}^{2}(x_{1})=0.
\end{array}\right.
\end{array}
\end{equation}
\begin{equation}
\begin{array}{llll}
   \left\{\begin{array}{l}
\vspace{2mm}\psi_{3}(x_{2},x_{3},x_{4})=x_{1},\\
\vspace{2mm}\psi_{3}(x_{2},x_{4},x_{5})=x_{2},\\
\vspace{2mm}\psi_{3}(x_{1},x_{4},x_{5})=x_{1},\\
\vspace{2mm}\psi_{3}(x_{2},x_{3},x_{1}^{*})=-x_{4}^{*},\\
\vspace{2mm}\psi_{3}(x_{2},x_{4},x_{1}^{*})=x_{3}^{*},\\
\vspace{2mm}\psi_{3}(x_{3},x_{4},x_{1}^{*})=-x_{2}^{*},\\
\vspace{2mm}\psi_{3}(x_{2},x_{4},x_{2}^{*})=-x_{5}^{*},\\
\vspace{2mm}\psi_{3}(x_{2},x_{5},x_{2}^{*})=x_{4}^{*},\\
\vspace{2mm}\psi_{3}(x_{4},x_{5},x_{2}^{*})=-x_{2}^{*},\\
\vspace{2mm}\psi_{3}(x_{1},x_{4},x_{1}^{*})=-x_{5}^{*},\\
\vspace{2mm}\psi_{3}(x_{4},x_{5},x_{1}^{*})=-x_{1}^{*},\\
\vspace{2mm}\psi_{3}(x_{1},x_{5},x_{1}^{*})=x_{4}^{*}.
\end{array}\right.
\end{array}
\begin{array}{llllllllll}
   \left\{\begin{array}{l}
\vspace{2mm}\widetilde{\Delta}^{3}(x_{1}^{*})=x_{1}^{*}\wedge x_{5}^{*}\wedge x_{4}^{*}+x_{2}^{*}\wedge x_{4}^{*}\wedge x_{3}^{*},\\
\vspace{2mm}\widetilde{\Delta}^{3}(x_{2}^{*})=x_{2}^{*}\wedge x_{5}^{*}\wedge x_{4}^{*},\\
\vspace{2mm}\widetilde{\Delta}^{3}(x_{1})=x_{1}\wedge x_{4}^{*}\wedge x_{5}^{*},\\
\vspace{2mm}\widetilde{\Delta}^{3}(x_{2})=x_{1}\wedge x_{3}^{*}\wedge x_{4}^{*}+x_{2}\wedge x_{4}^{*}\wedge x_{5}^{*},\\
\vspace{2mm}\widetilde{\Delta}^{3}(x_{3})=x_{1}\wedge x_{4}^{*}\wedge x_{2}^{*},\\
\vspace{2mm}\widetilde{\Delta}^{3}(x_{4})=x_{1}\wedge x_{2}^{*}\wedge x_{3}^{*}+x_{2}\wedge x_{5}^{*}\wedge x_{2}^{*}\\
\vspace{2mm}\hspace{1.5cm}+x_{5}^{*}\wedge x_{1}^{*}\wedge x_{1},\\
\vspace{2mm}\widetilde{\Delta}^{3}(x_{5})=x_{2}^{*}\wedge x_{4}^{*}\wedge x_{2}+x_{1}^{*}\wedge x_{4}^{*}\wedge x_{1},\\
\vspace{2mm}\widetilde{\Delta}^{3}(x_{3}^{*})=\widetilde{\Delta}^{3}(x_{4}^{*})=\widetilde{\Delta}^{3}(x_{5}^{*})=0.
\end{array}\right.
\end{array}
\end{equation}
\begin{equation}
\begin{array}{llll}
   \left\{\begin{array}{l}
\vspace{2mm}\psi_{4}(x_{2},x_{3},x_{4})=\alpha x_{1}+x_{2 },\\
\vspace{2mm}\psi_{4}(x_{1},x_{3},x_{4})=x_{2},\\
\vspace{2mm}\psi_{4}(x_{2},x_{4},x_{5})=x_{2},\\
\vspace{2mm}\psi_{4}(x_{1},x_{4},x_{5})=x_{1},\\
\vspace{2mm}\psi_{4}(x_{2},x_{3},x_{1}^{*})=-\alpha x_{4}^{*},\\
\vspace{2mm}\psi_{4}(x_{2},x_{4},x_{1}^{*})=\alpha x_{3}^{*},\\
\vspace{2mm}\psi_{4}(x_{3},x_{4},x_{1}^{*})=-\alpha x_{2}^{*},\\
\vspace{2mm}\psi_{4}(x_{2},x_{3},x_{2}^{*})=-x_{4}^{*},\\
\vspace{2mm}\psi_{4}(x_{2},x_{4},x_{2}^{*})=x_{3}^{*},\\
\vspace{2mm}\psi_{4}(x_{3},x_{4},x_{2}^{*})=-x_{2}^{*},\\
\vspace{2mm}\psi_{4}(x_{1},x_{3},x_{2}^{*})=-x_{4}^{*},\\
\vspace{2mm}\psi_{4}(x_{1},x_{4},x_{2}^{*})=x_{3}^{*},\\
\vspace{2mm}\psi_{4}(x_{3},x_{4},x_{2}^{*})=-x_{1}^{*},\\
\vspace{2mm}\psi_{4}(x_{2},x_{4},x_{2}^{*})=-x_{5}^{*},\\
\vspace{2mm}\psi_{4}(x_{4},x_{5},x_{2}^{*})=-x_{2}^{*},\\
\vspace{2mm}\psi_{4}(x_{2},x_{5},x_{2}^{*})=x_{4}^{*},\\
\vspace{2mm}\psi_{4}(x_{1},x_{4},x_{1}^{*})=-x_{5}^{*},\\
\vspace{2mm}\psi_{4}(x_{4},x_{5},x_{1}^{*})=-x_{1}^{*},\\
\vspace{2mm}\psi_{4}(x_{1},x_{5},x_{1}^{*})=x_{4}^{*}.
\end{array}\right.
\end{array}
\begin{array}{llllllllll}
   \left\{\begin{array}{l}
\vspace{2mm}\widetilde{\Delta}^{4}(x_{1}^{*})=x_{1}^{*}\wedge x_{5}^{*}\wedge x_{4}^{*}+\alpha x_{2}^{*}\wedge x_{4}^{*}\wedge x_{3}^{*},\\
\vspace{2mm}\widetilde{\Delta}^{4}(x_{2}^{*})=x_{2}^{*}\wedge x_{4}^{*}\wedge x_{3}^{*}+x_{1}^{*}\wedge x_{4}^{*}\wedge x_{3}^{*}\\
\vspace{2mm}\hspace{1.5cm}+x_{2}^{*}\wedge x_{5}^{*}\wedge x_{4}^{*},\\
\vspace{2mm}\widetilde{\Delta}^{4}(x_{1})=x_{2}\wedge x_{3}^{*}\wedge x_{4}^{*}+x_{1}\wedge x_{4}^{*}\wedge x_{5}^{*},\\
\vspace{2mm}\widetilde{\Delta}^{4}(x_{2})=\alpha x_{1}\wedge x_{3}^{*}\wedge x_{4}^{*}+x_{2}\wedge x_{3}^{*}\wedge x_{4}^{*}\\
\vspace{2mm}\hspace{1.5cm}+x_{2}\wedge x_{4}^{*}\wedge x_{5}^{*},\\
\vspace{2mm}\widetilde{\Delta}^{4}(x_{3})=\alpha x_{1}\wedge x_{4}^{*}\wedge x_{2}^{*}+x_{2}\wedge x_{4}^{*}\wedge x_{2}^{*}\\
\vspace{2mm}\hspace{1.5cm}+x_{2}\wedge x_{4}^{*}\wedge x_{1}^{*},\\
\vspace{2mm}\widetilde{\Delta}^{4}(x_{4})=\alpha x_{1}\wedge x_{2}^{*}\wedge x_{3}^{*}+x_{2}\wedge x_{2}^{*}\wedge x_{3}^{*}\\
\vspace{2mm}\hspace{1.5cm}+x_{1}^{*}\wedge x_{3}^{*}\wedge x_{2},+x_{5}^{*}\wedge x_{2}^{*}\wedge x_{2}\\
\vspace{2mm}\hspace{1.5cm}+x_{5}^{*}\wedge x_{1}^{*}\wedge x_{1},\\
\vspace{2mm}\widetilde{\Delta}^{4}(x_{5})=x_{2}^{*}\wedge x_{4}^{*}\wedge x_{2}+x_{1}^{*}\wedge x_{4}^{*}\wedge x_{1},\\
\vspace{2mm}\widetilde{\Delta}^{4}(x_{3}^{*})=\widetilde{\Delta}^{4}(x_{4}^{*})=\widetilde{\Delta}^{4}(x_{5}^{*})=0.\\
\end{array}\right.
\end{array}
\end{equation}
\begin{equation}
\begin{array}{llll}
   \left\{\begin{array}{l}
\vspace{2mm}\psi_{5}(x_{2},x_{3},x_{4})=\alpha x_{1}+x_{2},\\
\vspace{2mm}\psi_{5}(x_{1},x_{3},x_{4})=x_{2},\\
\vspace{2mm}\psi_{5}(x_{2},x_{3},x_{1}^{*})=-\alpha x_{4}^{*},\\
\vspace{2mm}\psi_{5}(x_{3},x_{4},x_{1}^{*})=-\alpha x_{2}^{*},\\
\vspace{2mm}\psi_{5}(x_{2},x_{4},x_{1}^{*})=\alpha x_{3}^{*},\\
\vspace{2mm}\psi_{5}(x_{2},x_{3},x_{2}^{*})=-x_{4}^{*},\\
\vspace{2mm}\psi_{5}(x_{2},x_{4},x_{2}^{*})=x_{3}^{*},\\
\vspace{2mm}\psi_{5}(x_{3},x_{4},x_{2}^{*})=-x_{2}^{*}-x_{1}^{*},\\
\vspace{2mm}\psi_{5}(x_{1},x_{3},x_{2}^{*})=-x_{4}^{*},\\
\vspace{2mm}\psi_{5}(x_{1},x_{4},x_{2}^{*})=x_{3}^{*}.
\end{array}\right.
\end{array}
\begin{array}{llllllllll}
   \left\{\begin{array}{l}
\vspace{2mm}\widetilde{\Delta}^{5}(x_{1}^{*})=\alpha x_{2}^{*}\wedge x_{4}^{*}\wedge x_{3}^{*},\\
\vspace{2mm}\widetilde{\Delta}^{5}(x_{2}^{*})=x_{1}^{*}\wedge x_{4}^{*}\wedge x_{3}^{*}+x_{2}^{*}\wedge x_{4}^{*}\wedge x_{3}^{*},\\
\vspace{2mm}\widetilde{\Delta}^{5}(x_{1})=\alpha x_{2}\wedge x_{3}^{*}\wedge x_{4}^{*},\\
\vspace{2mm}\widetilde{\Delta}^{5}(x_{2})=x_{1}\wedge x_{3}^{*}\wedge x_{4}^{*}+x_{2}\wedge x_{3}^{*}\wedge x_{4}^{*},\\
\vspace{2mm}\widetilde{\Delta}^{5}(x_{3})=x_{1}^{*}\wedge x_{2}\wedge x_{4}^{*}+x_{2}\wedge x_{4}^{*}\wedge x_{2}^{*}\\
\vspace{2mm}\hspace{1.5cm}+\alpha x_{1}\wedge x_{4}^{*}\wedge x_{2}^{*},\\
\vspace{2mm}\widetilde{\Delta}^{5}(x_{4})=x_{1}^{*}\wedge x_{3}^{*}\wedge x_{2}+\alpha x_{1}\wedge x_{2}^{*}\wedge x_{3}^{*}\\
\vspace{2mm}\hspace{1.5cm}+x_{2}\wedge x_{2}^{*}\wedge x_{3}^{*},\\
\vspace{2mm}\widetilde{\Delta}^{5}(x_{3}^{*})=\widetilde{\Delta}^{5}(x_{4}^{*})=0,\\
\vspace{2mm}\widetilde{\Delta}^{5}(x_{5}^{*})=\widetilde{\Delta}^{5}(x_{5})=0.
\end{array}\right.
\end{array}
\end{equation}
\begin{equation}
\begin{array}{llll}
   \left\{\begin{array}{l}
\vspace{2mm}\psi_{6}(x_{2},x_{3},x_{4})=x_{1},\\
\vspace{2mm}\psi_{6}(x_{1},x_{3},x_{4})=x_{2},\\
\vspace{2mm}\psi_{6}(x_{2},x_{4},x_{5})=x_{2},\\
\vspace{2mm}\psi_{6}(x_{1},x_{4},x_{5})=x_{1},\\
\vspace{2mm}\psi_{6}(x_{2},x_{3},x_{1}^{*})=-x_{4}^{*},\\
\vspace{2mm}\psi_{6}(x_{2},x_{4},x_{1}^{*})=x_{3}^{*},\\
\vspace{2mm}\psi_{6}(x_{3},x_{4},x_{1}^{*})=-x_{2}^{*},\\
\vspace{2mm}\psi_{6}(x_{1},x_{3},x_{2}^{*})=-x_{4}^{*},\\
\vspace{2mm}\psi_{6}(x_{3},x_{4},x_{2}^{*})=-x_{1}^{*},\\
\vspace{2mm}\psi_{6}(x_{1},x_{4},x_{2}^{*})=x_{3}^{*},\\
\vspace{2mm}\psi_{6}(x_{2},x_{4},x_{2}^{*})=-x_{5}^{*},\\
\vspace{2mm}\psi_{6}(x_{4},x_{5},x_{2}^{*})=-x_{2}^{*},\\
\vspace{2mm}\psi_{6}(x_{2},x_{5},x_{2}^{*})=x_{4}^{*},\\
\vspace{2mm}\psi_{6}(x_{1},x_{4},x_{1}^{*})=-x_{5}^{*},\\
\vspace{2mm}\psi_{6}(x_{4},x_{5},x_{1}^{*})=-x_{1}^{*},\\
\vspace{2mm}\psi_{6}(x_{1},x_{5},x_{2}^{*})=x_{4}^{*}.
\end{array}\right.
\end{array}
\begin{array}{llllllllll}
   \left\{\begin{array}{l}
\vspace{2mm}\widetilde{\Delta}^{6}(x_{1}^{*})=x_{2}^{*}\wedge x_{4}^{*}\wedge x_{3}^{*}+ x_{1}^{*}\wedge x_{5}^{*}\wedge x_{4}^{*},\\
\vspace{2mm}\widetilde{\Delta}^{6}(x_{2}^{*})=x_{1}^{*}\wedge x_{4}^{*}\wedge x_{3}^{*}+x_{2}^{*}\wedge x_{5}^{*}\wedge x_{4}^{*},\\
\vspace{2mm}\widetilde{\Delta}^{6}(x_{1})=x_{1}\wedge x_{4}^{*}\wedge x_{5}^{*}+x_{2}\wedge x_{3}^{*}\wedge x_{4}^{*},\\
\vspace{2mm}\widetilde{\Delta}^{6}(x_{2})=x_{1}\wedge x_{3}^{*}\wedge x_{4}^{*}+x_{2}\wedge x_{4}^{*}\wedge x_{5}^{*},\\
\vspace{2mm}\widetilde{\Delta}^{6}(x_{3})=x_{1}\wedge x_{4}^{*}\wedge x_{2}^{*}+x_{2}\wedge x_{4}^{*}\wedge x_{1}^{*},\\
\vspace{2mm}\widetilde{\Delta}^{6}(x_{4})=x_{1}\wedge x_{2}^{*}\wedge x_{3}^{*}+x_{2}\wedge x_{1}^{*}\wedge x_{3}^{*},\\
\vspace{2mm}\hspace{1.5cm} +x_{5}^{*}\wedge x_{2}^{*}\wedge x_{2}+x_{5}^{*}\wedge x_{1}^{*}\wedge x_{1},\\
\vspace{2mm}\widetilde{\Delta}^{6}(x_{5})=x_{2}^{*}\wedge x_{4}^{*}\wedge x_{2}+x_{1}^{*}\wedge x_{4}^{*}\wedge x_{1},\\
\vspace{2mm}\widetilde{\Delta}^{6}(x_{3}^{*})=\widetilde{\Delta}^{6}(x_{4}^{*})=\widetilde{\Delta}^{6}(x_{5}^{*})=0.\\
\end{array}\right.
\end{array}
\end{equation}
\begin{equation}
\begin{array}{llll}
   \left\{\begin{array}{l}
\vspace{2mm}\psi_{7}(x_{2},x_{3},x_{4})=x_{1},\\
\vspace{2mm}\psi_{7}(x_{2},x_{4},x_{5})=x_{2},\\
\vspace{2mm}\psi_{7}(x_{3},x_{4},x_{5})=x_{3},\\
\vspace{2mm}\psi_{7}(x_{2},x_{3},x_{1}^{*})=-x_{4}^{*},\\
\vspace{2mm}\psi_{7}(x_{2},x_{4},x_{1}^{*})= x_{3}^{*},\\
\vspace{2mm}\psi_{7}(x_{3},x_{4},x_{1}^{*})=-x_{2}^{*},\\
\vspace{2mm}\psi_{7}(x_{2},x_{4},x_{2}^{*})=-x_{5}^{*},\\
\vspace{2mm}\psi_{7}(x_{4},x_{5},x_{2}^{*})=-x_{2}^{*},\\
\vspace{2mm}\psi_{7}(x_{2},x_{5},x_{2}^{*})=x_{4}^{*},\\
\vspace{2mm}\psi_{7}(x_{3},x_{4},x_{3}^{*})=-x_{5}^{*},\\
\vspace{2mm}\psi_{7}(x_{4},x_{5},x_{3}^{*})=-x_{3}^{*},\\
\vspace{2mm}\psi_{7}(x_{3},x_{5},x_{3}^{*})=x_{4}^{*}.
\end{array}\right.
\end{array}
\begin{array}{llllllllll}
   \left\{\begin{array}{l}
\vspace{2mm}\widetilde{\Delta}^{7}(x_{1}^{*})=x_{2}^{*}\wedge x_{4}^{*}\wedge x_{3}^{*},\\
\vspace{2mm}\widetilde{\Delta}^{7}(x_{2}^{*})=x_{2}^{*}\wedge x_{5}^{*}\wedge x_{4}^{*},\\
\vspace{2mm}\widetilde{\Delta}^{7}(x_{3}^{*})=x_{3}^{*}\wedge x_{5}^{*}\wedge x_{4}^{*},\\
\vspace{2mm}\widetilde{\Delta}^{7}(x_{2})=x_{1}\wedge x_{3}^{*}\wedge x_{4}^{*}+x_{2}\wedge x_{4}^{*}\wedge x_{5}^{*},\\
\vspace{2mm}\widetilde{\Delta}^{7}(x_{3})=x_{1}\wedge x_{4}^{*}\wedge x_{2}^{*}+x_{3}\wedge x_{4}^{*}\wedge x_{5}^{*},\\
\vspace{2mm}\widetilde{\Delta}^{7}(x_{4})=x_{1}\wedge x_{2}^{*}\wedge x_{3}^{*}+x_{2}\wedge x_{5}^{*}\wedge x_{2}^{*}\\
\vspace{2mm}\hspace{1.5cm}+x_{3}\wedge x_{5}^{*}\wedge x_{3}^{*},\\
\vspace{2mm}\widetilde{\Delta}^{7}(x_{5})=x_{2}\wedge x_{2}^{*}\wedge x_{4}^{*}+x_{3}\wedge x_{3}^{*}\wedge x_{4}^{*},\\
\vspace{2mm}\widetilde{\Delta}^{7}(x_{4}^{*})=\widetilde{\Delta}^{7}(x_{5}^{*})=\widetilde{\Delta}^{7}(x_{1})=0.
\end{array}\right.
\end{array}
\end{equation}
\begin{equation}
\begin{array}{llll}
   \left\{\begin{array}{l}
\vspace{2mm}\psi_{8}(x_{2},x_{3},x_{4})=x_{2},\\
\vspace{2mm}\psi_{8}(x_{1},x_{3},x_{4})=x_{1},\\
\vspace{2mm}\psi_{8}(x_{1},x_{3},x_{1}^{*})=-x_{4}^{*},\\
\vspace{2mm}\psi_{8}(x_{1},x_{4},x_{1}^{*})= x_{3}^{*},\\
\vspace{2mm}\psi_{8}(x_{3},x_{4},x_{1}^{*})=-x_{1}^{*},\\
\vspace{2mm}\psi_{8}(x_{2},x_{3},x_{2}^{*})=-x_{4}^{*},\\
\vspace{2mm}\psi_{8}(x_{2},x_{4},x_{2}^{*})=x_{3}^{*},\\
\vspace{2mm}\psi_{8}(x_{3},x_{4},x_{2}^{*})=-x_{2}^{*}.
\end{array}\right.
\end{array}
\begin{array}{llllllllll}
   \left\{\begin{array}{l}
\vspace{2mm}\widetilde{\Delta}^{8}(x_{1}^{*})= x_{1}^{*}\wedge x_{4}^{*}\wedge x_{3}^{*},\\
\vspace{2mm}\widetilde{\Delta}^{8}(x_{1})=x_{1}\wedge x_{3}^{*}\wedge x_{4}^{*},\\
\vspace{2mm}\widetilde{\Delta}^{8}(x_{2})=x_{2}\wedge x_{3}^{*}\wedge x_{4}^{*},\\
\vspace{2mm}\widetilde{\Delta}^{8}(x_{3})=x_{1}\wedge x_{4}^{*}\wedge x_{1}^{*}+x_{2}\wedge x_{4}^{*}\wedge x_{2}^{*},\\
\vspace{2mm}\widetilde{\Delta}^{8}(x_{4})=x_{1}\wedge x_{1}^{*}\wedge x_{3}^{*}+x_{2}\wedge x_{2}^{*}\wedge x_{3}^{*},\\
\vspace{2mm}\widetilde{\Delta}^{8}(x_{2}^{*})=x_{2}^{*}\wedge x_{4}^{*}\wedge x_{3}^{*},\\
\vspace{2mm}\widetilde{\Delta}^{8}(x_{3}^{*})=\widetilde{\Delta}^{8}(x_{4}^{*})=0,\\
\vspace{2mm}\widetilde{\Delta}^{8}(x_{5}^{*})=\widetilde{\Delta}^{8}(x_{5})=0.
\end{array}\right.
\end{array}
\end{equation}
\begin{equation}
\begin{array}{llll}
   \left\{\begin{array}{l}
\vspace{2mm}\psi_{9}(x_{2},x_{3},x_{4})=x_{1},\\
\vspace{2mm}\psi_{9}(x_{3},x_{4},x_{5})=x_{3}+\alpha x_{2},\\
\vspace{2mm}\psi_{9}(x_{2},x_{4},x_{5})=x_{3},\\
\vspace{2mm}\psi_{9}(x_{1},x_{4},x_{5})=x_{1},\\
\vspace{2mm}\psi_{9}(x_{2},x_{3},x_{1}^{*})=-x_{4}^{*},\\
\vspace{2mm}\psi_{9}(x_{2},x_{4},x_{1}^{*})= x_{3}^{*},\\
\vspace{2mm}\psi_{9}(x_{3},x_{4},x_{1}^{*})=-x_{2}^{*},\\
\vspace{2mm}\psi_{9}(x_{3},x_{4},x_{3}^{*})=-x_{5}^{*},\\
\vspace{2mm}\psi_{9}(x_{4},x_{5},x_{3}^{*})=-x_{3}^{*},\\
\vspace{2mm}\psi_{9}(x_{3},x_{5},x_{3}^{*})=x_{4}^{*},\\
\vspace{2mm}\psi_{9}(x_{3},x_{4},x_{2}^{*})=-\alpha x_{5}^{*},\\
\vspace{2mm}\psi_{9}(x_{4},x_{5},x_{2}^{*})=-\alpha x_{3}^{*},\\
\vspace{2mm}\psi_{9}(x_{3},x_{5},x_{2}^{*})=\alpha x_{4}^{*},\\
\vspace{2mm}\psi_{9}(x_{2},x_{4},x_{3}^{*})=-x_{5}^{*},\\
\vspace{2mm}\psi_{9}(x_{4},x_{5},x_{3}^{*})=-x_{2}^{*},\\
\vspace{2mm}\psi_{9}(x_{2},x_{5},x_{3}^{*})=x_{4}^{*},\\
\vspace{2mm}\psi_{9}(x_{1},x_{4},x_{1}^{*})=-x_{5}^{*},\\
\vspace{2mm}\psi_{9}(x_{4},x_{5},x_{1}^{*})=-x_{1}^{*},\\
\vspace{2mm}\psi_{9}(x_{1},x_{5},x_{1}^{*})=x_{4}^{*}.
\end{array}\right.
\end{array}
\begin{array}{llllllllll}
   \left\{\begin{array}{l}
\vspace{2mm}\widetilde{\Delta}^{9}(x_{1}^{*})=x_{3}^{*}\wedge x_{2}^{*}\wedge x_{4}^{*}+x_{4}^{*}\wedge x_{1}^{*}\wedge x_{5}^{*},\\
\vspace{2mm}\widetilde{\Delta}^{9}(x_{2}^{*})=\alpha x_{4}^{*}\wedge x_{3}^{*}\wedge x_{5}^{*},\\
\vspace{2mm}\widetilde{\Delta}^{9}(x_{3}^{*})=x_{4}^{*}\wedge x_{3}^{*}\wedge x_{5}^{*}+x_{4}^{*}\wedge x_{2}^{*}\wedge x_{5}^{*},\\
\vspace{2mm}\widetilde{\Delta}^{9}(x_{1})=x_{1}\wedge x_{4}^{*}\wedge x_{5}^{*},\\
\vspace{2mm}\widetilde{\Delta}^{9}(x_{2})=x_{1}\wedge x_{3}^{*}\wedge x_{4}^{*}+x_{3}\wedge x_{4}^{*}\wedge x_{5}^{*},\\
\vspace{2mm}\widetilde{\Delta}^{9}(x_{3})=x_{1}\wedge x_{4}^{*}\wedge x_{2}^{*}+\alpha x_{2}\wedge x_{4}^{*}\wedge x_{5}^{*}\\
\vspace{2mm}\hspace{1.5cm}+x_{3}\wedge x_{4}^{*}\wedge x_{5}^{*},\\
\vspace{2mm}\widetilde{\Delta}^{9}(x_{4})=x_{1}\wedge x_{2}^{*}\wedge x_{3}^{*}+x_{3}\wedge x_{5}^{*}\wedge x_{3}^{*}\\
\vspace{2mm}\hspace{1.5cm}+\alpha x_{2}\wedge x_{5}^{*}\wedge x_{3}^{*}+x_{3}\wedge x_{5}^{*}\wedge x_{2}^{*}\\
\vspace{2mm}\hspace{1.5cm}+x_{1}\wedge x_{5}^{*}\wedge x_{1}^{*},\\
\vspace{2mm}\widetilde{\Delta}^{9}(x_{5})= x_{3}^{*}\wedge x_{4}^{*}\wedge x_{3}+\alpha x_{3}^{*}\wedge x_{4}^{*}\wedge x_{2}\\
\vspace{2mm}\hspace{1.5cm}+ x_{2}^{*}\wedge x_{4}^{*}\wedge x_{3}+ x_{1}^{*}\wedge x_{4}^{*}\wedge x_{1},\\
\vspace{2mm}\widetilde{\Delta}^{9}(x_{4}^{*})=\widetilde{\Delta}^{9}(x_{5}^{*})=0.\\
\end{array}\right.
\end{array}
\end{equation}
\begin{equation}
\begin{array}{llll}
   \left\{\begin{array}{l}
\vspace{2mm}\psi_{10}(x_{2},x_{3},x_{4})=x_{1},\\
\vspace{2mm}\psi_{10}(x_{3},x_{4},x_{5})=x_{3},\\
\vspace{2mm}\psi_{10}(x_{2},x_{4},x_{5})=x_{2},\\
\vspace{2mm}\psi_{10}(x_{1},x_{4},x_{5})=2x_{2},\\
\vspace{2mm}\psi_{10}(x_{2},x_{3},x_{1}^{*})=-x_{4}^{*},\\
\vspace{2mm}\psi_{10}(x_{3},x_{4},x_{1}^{*})=-x_{2}^{*},\\
\vspace{2mm}\psi_{10}(x_{2},x_{4},x_{1}^{*})=x_{3}^{*},\\
\vspace{2mm}\psi_{10}(x_{3},x_{4},x_{3}^{*})=-x_{5}^{*},\\
\vspace{2mm}\psi_{10}(x_{4},x_{5},x_{3}^{*})=-x_{3}^{*},\\
\vspace{2mm}\psi_{10}(x_{3},x_{5},x_{3}^{*})=x_{4}^{*},\\
\vspace{2mm}\psi_{10}(x_{2},x_{4},x_{2}^{*})=-x_{5}^{*},\\
\vspace{2mm}\psi_{10}(x_{4},x_{5},x_{2}^{*})=-x_{2}^{*},\\
\vspace{2mm}\psi_{10}(x_{2},x_{5},x_{2}^{*})=x_{4}^{*},\\
\vspace{2mm}\psi_{10}(x_{1},x_{4},x_{2}^{*})=-2x_{5}^{*},\\
\vspace{2mm}\psi_{10}(x_{4},x_{5},x_{2}^{*})=-2x_{1}^{*},\\
\vspace{2mm}\psi_{10}(x_{1},x_{5},x_{2}^{*})=2x_{4}^{*}.
\end{array}\right.
\end{array}
\begin{array}{llllllllll}
   \left\{\begin{array}{l}
\vspace{2mm}\widetilde{\Delta}^{10}(x_{1}^{*})=x_{3}^{*}\wedge x_{2}^{*}\wedge x_{4}^{*},\\
\vspace{2mm}\widetilde{\Delta}^{10}(x_{2}^{*})=x_{4}^{*}\wedge x_{2}^{*}\wedge x_{5}^{*}+2x_{4}^{*}\wedge x_{1}^{*}\wedge x_{5}^{*},\\
\vspace{2mm}\widetilde{\Delta}^{10}(x_{3}^{*})=x_{4}^{*}\wedge x_{3}^{*}\wedge x_{5}^{*},\\
\vspace{2mm}\widetilde{\Delta}^{10}(x_{1})=2x_{2}\wedge x_{4}^{*}\wedge x_{5}^{*}\\
\vspace{2mm}\widetilde{\Delta}^{10}(x_{2})=x_{1}\wedge x_{3}^{*}\wedge x_{4}^{*}+x_{2}\wedge x_{4}^{*}\wedge x_{5}^{*},\\
\vspace{2mm}\widetilde{\Delta}^{10}(x_{3})=x_{1}\wedge x_{4}^{*}\wedge x_{2}^{*}+ x_{3}\wedge x_{4}^{*}\wedge x_{5}^{*},\\
\vspace{2mm}\widetilde{\Delta}^{10}(x_{4})=x_{1}\wedge x_{2}^{*}\wedge x_{3}^{*}+x_{3}\wedge x_{5}^{*}\wedge x_{3}^{*},\\
\vspace{2mm}\hspace{1.5cm}+x_{2}\wedge x_{5}^{*}\wedge x_{2}^{*}+2x_{2}\wedge x_{5}^{*}\wedge x_{1}^{*}\\
\vspace{2mm}\widetilde{\Delta}^{10}(x_{5})=x_{3}\wedge x_{3}^{*}\wedge x_{4}^{*}+x_{2}\wedge x_{2}^{*}\wedge x_{4}^{*}\\
\vspace{2mm}\hspace{1.5cm}+2x_{2}\wedge x_{1}^{*}\wedge x_{4}^{*},\\
\vspace{2mm}\widetilde{\Delta}^{10}(x_{4}^{*})=\widetilde{\Delta}^{10}(x_{5}^{*})=0.\\
\end{array}\right.
\end{array}
\end{equation}
\begin{equation}
\begin{array}{llll}
   \left\{\begin{array}{l}
\vspace{2mm}\psi_{11}(x_{2},x_{3},x_{4})=x_{1},\\
\vspace{2mm}\psi_{11}(x_{1},x_{3},x_{4})=x_{2},\\
\vspace{2mm}\psi_{11}(x_{1},x_{2},x_{4})=x_{3},\\
\vspace{2mm}\psi_{11}(x_{2},x_{3},x_{1}^{*})=-x_{4}^{*},\\
\vspace{2mm}\psi_{11}(x_{3},x_{4},x_{1}^{*})=-x_{2}^{*},\\
\vspace{2mm}\psi_{11}(x_{2},x_{4},x_{1}^{*})=x_{3}^{*},\\
\vspace{2mm}\psi_{11}(x_{1},x_{3},x_{2}^{*})=-x_{4}^{*},\\
\vspace{2mm}\psi_{11}(x_{3},x_{4},x_{2}^{*})=-x_{1}^{*},\\
\vspace{2mm}\psi_{11}(x_{1},x_{4},x_{2}^{*})=x_{3}^{*},\\
\vspace{2mm}\psi_{11}(x_{1},x_{4},x_{3}^{*})=x_{2}^{*},\\
\vspace{2mm}\psi_{11}(x_{1},x_{2},x_{3}^{*})=-x_{4}^{*},\\
\vspace{2mm}\psi_{11}(x_{2},x_{4},x_{3}^{*})=-x_{1}^{*}.
\end{array}\right.
\end{array}
\begin{array}{llllllllll}
   \left\{\begin{array}{l}
\vspace{2mm}\widetilde{\Delta}^{11}(x_{1}^{*})=x_{3}^{*}\wedge x_{2}^{*}\wedge x_{4}^{*},\\
\vspace{2mm}\widetilde{\Delta}^{11}(x_{2}^{*})=x_{3}^{*}\wedge x_{1}^{*}\wedge x_{4}^{*},\\
\vspace{2mm}\widetilde{\Delta}^{11}(x_{3}^{*})=x_{2}^{*}\wedge x_{1}^{*}\wedge x_{4}^{*},\\
\vspace{2mm}\widetilde{\Delta}^{11}(x_{1})=x_{2}\wedge x_{3}^{*}\wedge x_{4}^{*}+x_{3}\wedge x_{2}^{*}\wedge x_{4}^{*},\\
\vspace{2mm}\widetilde{\Delta}^{11}(x_{2})=x_{1}\wedge x_{3}^{*}\wedge x_{4}^{*}+x_{3}\wedge x_{4}^{*}\wedge x_{1}^{*},\\
\vspace{2mm}\widetilde{\Delta}^{11}(x_{3})=x_{1}\wedge x_{4}^{*}\wedge x_{2}^{*}+ x_{2}\wedge x_{4}^{*}\wedge x_{1}^{*},\\
\vspace{2mm}\widetilde{\Delta}^{11}(x_{4})=x_{1}\wedge x_{2}^{*}\wedge x_{3}^{*}+x_{2}\wedge x_{1}^{*}\wedge x_{3}^{*}\\
\vspace{2mm}\hspace{1.5cm}+x_{3}\wedge x_{1}^{*}\wedge x_{2}^{*},\\
\vspace{2mm}\widetilde{\Delta}^{11}(x_{4}^{*})=\widetilde{\Delta}^{11}(x_{5}^{*})=0,\\
\vspace{2mm}\widetilde{\Delta}^{11}(x_{5})=0.
\end{array}\right.
\end{array}
\end{equation}
\begin{equation}
\begin{array}{llll}
   \left\{\begin{array}{l}
\vspace{2mm}\psi_{12}(x_{1},x_{2},x_{5})=x_{1},\\
\vspace{2mm}\psi_{12}(x_{2},x_{4},x_{5})=x_{3},\\
\vspace{2mm}\psi_{12}(x_{3},x_{4},x_{5})=\beta x_{2}\\
\vspace{2mm}\hspace{1.5cm}+(1+\beta)x_{3},\\
\vspace{2mm}\psi_{12}(x_{1},x_{4},x_{1}^{*})=-x_{5}^{*},\\
\vspace{2mm}\psi_{12}(x_{4},x_{5},x_{1}^{*})=-x_{1}^{*},\\
\vspace{2mm}\psi_{12}(x_{1},x_{5},x_{1}^{*})=x_{4}^{*},\\
\vspace{2mm}\psi_{12}(x_{2},x_{4},x_{3}^{*})=-x_{5}^{*},\\
\vspace{2mm}\psi_{12}(x_{4},x_{5},x_{3}^{*})=-x_{2}^{*},\\
\vspace{2mm}\psi_{12}(x_{2},x_{5},x_{3}^{*})=x_{4}^{*},\\
\vspace{2mm}\psi_{12}(x_{3},x_{4},x_{2}^{*})=-\beta x_{5}^{*},\\
\vspace{2mm}\psi_{12}(x_{4},x_{5},x_{2}^{*})=-\beta x_{3}^{*},\\
\vspace{2mm}\psi_{12}(x_{3},x_{5},x_{2}^{*})=\beta x_{4}^{*},\\
\vspace{2mm}\psi_{12}(x_{3},x_{4},x_{3}^{*})=-(1+\beta)x_{5}^{*},\\
\vspace{2mm}\psi_{12}(x_{4},x_{5},x_{3}^{*})=-(1+\beta)x_{3}^{*},\\
\vspace{2mm}\psi_{12}(x_{3},x_{5},x_{3}^{*})=(1+\beta)x_{4}^{*}.
\end{array}\right.
\end{array}
\begin{array}{llllllllll}
   \left\{\begin{array}{l}
\vspace{2mm}\widetilde{\Delta}^{12}(x_{1}^{*})=x_{4}^{*}\wedge x_{1}^{*}\wedge x_{5}^{*},\\
\vspace{2mm}\widetilde{\Delta}^{12}(x_{2}^{*})=\beta x_{4}^{*}\wedge x_{3}^{*}\wedge x_{5}^{*},\\
\vspace{2mm}\widetilde{\Delta}^{12}(x_{3}^{*})=x_{4}^{*}\wedge x_{2}^{*}\wedge x_{5}^{*}\\
\vspace{2mm}\hspace{1.5cm}+(1+\beta)x_{4}^{*}\wedge x_{3}^{*}\wedge x_{5}^{*},\\
\vspace{2mm}\widetilde{\Delta}^{12}(x_{1})=x_{1}\wedge x_{4}^{*}\wedge x_{5}^{*},\\
\vspace{2mm}\widetilde{\Delta}^{12}(x_{2})=x_{3}\wedge x_{4}^{*}\wedge x_{5}^{*},\\
\vspace{2mm}\widetilde{\Delta}^{12}(x_{3})=\beta x_{2}\wedge x_{4}^{*}\wedge x_{5}^{*}\\
\vspace{2mm}\hspace{1.5cm}+ (1+\beta)x_{3}\wedge x_{4}^{*}\wedge x_{5}^{*},\\
\vspace{2mm}\widetilde{\Delta}^{12}(x_{4})=x_{1}\wedge x_{5}^{*}\wedge x_{1}^{*}+x_{3}\wedge x_{5}^{*}\wedge x_{2}^{*}\\
\vspace{2mm}\hspace{1.5cm}+\beta x_{2}\wedge x_{5}^{*}\wedge x_{3}^{*}\\
\vspace{2mm}\hspace{1.5cm}+(1+\beta)x_{3}\wedge x_{5}^{*}\wedge x_{3}^{*},\\
\vspace{2mm}\widetilde{\Delta}^{12}(x_{5})=x_{1}\wedge x_{1}^{*}\wedge x_{4}^{*}+x_{3}\wedge x_{2}^{*}\wedge x_{4}^{*}\\
\vspace{2mm}\hspace{1.5cm}+\beta x_{2}\wedge x_{3}^{*}\wedge x_{4}^{*}\\
\vspace{2mm}\hspace{1.5cm}+(1+\beta)x_{3}\wedge x_{3}^{*}\wedge x_{4}^{*},\\
\vspace{2mm}\widetilde{\Delta}^{12}(x_{4}^{*})=\widetilde{\Delta}^{12}(x_{5}^{*})=0.\\
\end{array}\right.
\end{array}
\end{equation}
\begin{equation}
\begin{array}{llll}
   \left\{\begin{array}{l}
\vspace{2mm}\psi_{13}(x_{1},x_{4},x_{5})=x_{1},\\
\vspace{2mm}\psi_{13}(x_{2},x_{4},x_{5})=x_{2},\\
\vspace{2mm}\psi_{13}(x_{3},x_{4},x_{5})=x_{3},\\
\vspace{2mm}\psi_{13}(x_{1},x_{4},x_{1}^{*})=-x_{5}^{*},\\
\vspace{2mm}\psi_{13}(x_{4},x_{5},x_{1}^{*})=-x_{1}^{*},\\
\vspace{2mm}\psi_{13}(x_{1},x_{5},x_{1}^{*})=x_{4}^{*},\\
\vspace{2mm}\psi_{13}(x_{2},x_{4},x_{2}^{*})=-x_{5}^{*},\\
\vspace{2mm}\psi_{13}(x_{4},x_{5},x_{2}^{*})=-x_{2}^{*},\\
\vspace{2mm}\psi_{13}(x_{2},x_{5},x_{2}^{*})=x_{4}^{*},\\
\vspace{2mm}\psi_{13}(x_{3},x_{4},x_{3}^{*})=-x_{5}^{*},\\
\vspace{2mm}\psi_{13}(x_{4},x_{5},x_{3}^{*})=-x_{3}^{*},\\
\vspace{2mm}\psi_{13}(x_{3},x_{5},x_{3}^{*})=x_{4}^{*}.
\end{array}\right.
\end{array}
\begin{array}{llllllllll}
   \left\{\begin{array}{l}
\vspace{2mm}\widetilde{\Delta}^{13}(x_{1}^{*})=x_{4}^{*}\wedge x_{1}^{*}\wedge x_{5}^{*},\\
\vspace{2mm}\widetilde{\Delta}^{13}(x_{2}^{*})=x_{4}^{*}\wedge x_{2}^{*}\wedge x_{5}^{*},\\
\vspace{2mm}\widetilde{\Delta}^{13}(x_{3}^{*})=x_{4}^{*}\wedge x_{3}^{*}\wedge x_{5}^{*},\\
\vspace{2mm}\widetilde{\Delta}^{13}(x_{1})=x_{1}\wedge x_{4}^{*}\wedge x_{5}^{*},\\
\vspace{2mm}\widetilde{\Delta}^{13}(x_{2})=x_{2}\wedge x_{4}^{*}\wedge x_{5}^{*},\\
\vspace{2mm}\widetilde{\Delta}^{13}(x_{3})=x_{3}\wedge x_{4}^{*}\wedge x_{5}^{*},\\
\vspace{2mm}\widetilde{\Delta}^{13}(x_{4})= x_{5}^{*}\wedge x_{1}^{*}\wedge x_{1}+ x_{5}^{*}\wedge x_{2}^{*}\wedge x_{2}\\
\vspace{2mm}\hspace{1.5cm}+ x_{5}^{*}\wedge x_{3}^{*}\wedge x_{3},\\
\vspace{2mm}\widetilde{\Delta}^{13}(x_{5})=x_{1}^{*}\wedge x_{4}^{*}\wedge x_{1} +x_{2}^{*}\wedge x_{4}^{*}\wedge x_{2}\\
\vspace{2mm}\hspace{1.5cm}+ x_{3}^{*}\wedge x_{4}^{*}\wedge x_{3},\\
\vspace{2mm}\widetilde{\Delta}^{13}(x_{4}^{*})=\widetilde{\Delta}^{13}(x_{5}^{*})=0.
\end{array}\right.
\end{array}
\end{equation}
\begin{equation}
\begin{array}{llll}
   \left\{\begin{array}{l}
\vspace{2mm}\psi_{14}(x_{1},x_{4},x_{5})=x_{2},\\
\vspace{2mm}\psi_{14}(x_{2},x_{4},x_{5})=x_{3},\\
\vspace{2mm}\psi_{14}(x_{3},x_{4},x_{5})=sx_{1}\\
\vspace{2mm}+tx_{2}+ux_{3},\\
\vspace{2mm}\psi_{14}(x_{1},x_{4},x_{2}^{*})=-x_{5}^{*},\\
\vspace{2mm}\psi_{14}(x_{4},x_{5},x_{2}^{*})=-x_{1}^{*},\\
\vspace{2mm}\psi_{14}(x_{1},x_{5},x_{2}^{*})=x_{4}^{*},\\
\vspace{2mm}\psi_{14}(x_{2},x_{4},x_{3}^{*})=-x_{5}^{*},\\
\vspace{2mm}\psi_{14}(x_{4},x_{5},x_{3}^{*})=-x_{2}^{*},\\
\vspace{2mm}\psi_{14}(x_{2},x_{5},x_{3}^{*})=x_{4}^{*},\\
\vspace{2mm}\psi_{14}(x_{3},x_{4},x_{1}^{*})=-sx_{5}^{*},\\
\vspace{2mm}\psi_{14}(x_{4},x_{5},x_{1}^{*})=-sx_{3}^{*},\\
\vspace{2mm}\psi_{14}(x_{3},x_{5},x_{1}^{*})=sx_{4}^{*},\\
\vspace{2mm}\psi_{14}(x_{3},x_{4},x_{2}^{*})=-tx_{5}^{*},\\
\vspace{2mm}\psi_{14}(x_{4},x_{5},x_{2}^{*})=-tx_{3}^{*},\\
\vspace{2mm}\psi_{14}(x_{3},x_{5},x_{2}^{*})=tx_{4}^{*},\\
\vspace{2mm}\psi_{14}(x_{3},x_{4},x_{3}^{*})=-ux_{5}^{*},\\
\vspace{2mm}\psi_{14}(x_{4},x_{5},x_{3}^{*})=-ux_{3}^{*},\\
\vspace{2mm}\psi_{14}(x_{3},x_{5},x_{3}^{*})=ux_{4}^{*},\\
\end{array}\right.
\end{array}
\begin{array}{llllllllll}
   \left\{\begin{array}{l}
\vspace{2mm}\widetilde{\Delta}^{14}(x_{1}^{*})=sx_{4}^{*}\wedge x_{3}^{*}\wedge x_{5}^{*},\\
\vspace{2mm}\widetilde{\Delta}^{14}(x_{2}^{*})=x_{4}^{*}\wedge x_{1}^{*}\wedge x_{5}^{*}+tx_{4}^{*}\wedge x_{3}^{*}\wedge x_{5}^{*},\\
\vspace{2mm}\widetilde{\Delta}^{14}(x_{3}^{*})=x_{4}^{*}\wedge x_{2}^{*}\wedge x_{5}^{*}+ux_{4}^{*}\wedge x_{3}^{*}\wedge x_{5}^{*},\\
\vspace{2mm}\widetilde{\Delta}^{14}(x_{1})=x_{2}\wedge x_{4}^{*}\wedge x_{5}^{*},\\
\vspace{2mm}\widetilde{\Delta}^{14}(x_{2})=x_{3}\wedge x_{4}^{*}\wedge x_{5}^{*},\\
\vspace{2mm}\widetilde{\Delta}^{14}(x_{3})=sx_{1}\wedge x_{4}^{*}\wedge x_{5}^{*}+tx_{2}\wedge x_{4}^{*}\wedge x_{5}^{*}\\
\vspace{2mm}\hspace{1.5cm}+ux_{3}\wedge x_{4}^{*}\wedge x_{5}^{*},\\
\vspace{2mm}\widetilde{\Delta}^{14}(x_{4})=x_{2}\wedge x_{5}^{*}\wedge x_{1}^{*}+x_{3}\wedge x_{5}^{*}\wedge x_{2}^{*}\\
\vspace{2mm}\hspace{1.5cm}+sx_{1}\wedge x_{5}^{*}\wedge x_{3}^{*}+tx_{2}\wedge x_{5}^{*}\wedge x_{3}^{*}\\
\vspace{2mm}\hspace{1.5cm}+ux_{3}\wedge x_{5}^{*}\wedge x_{3}^{*},\\
\vspace{2mm}\widetilde{\Delta}^{14}(x_{5})=x_{2}\wedge x_{1}^{*}\wedge x_{4}^{*}+x_{3}\wedge x_{2}^{*}\wedge x_{4}^{*}\\
\vspace{2mm}\hspace{1.5cm}+sx_{1}\wedge x_{3}^{*}\wedge x_{4}^{*}+tx_{2}\wedge x_{3}^{*}\wedge x_{4}^{*}\\
\vspace{2mm}\hspace{1.5cm}+ux_{3}\wedge x_{3}^{*}\wedge x_{4}^{*},\\
\vspace{2mm}\widetilde{\Delta}^{14}(x_{4}^{*})=\widetilde{\Delta}^{14}(x_{5}^{*})=0.\\
\end{array}\right.
\end{array}
\end{equation}
\begin{equation}
\begin{array}{llll}
   \left\{\begin{array}{l}
\vspace{2mm}\psi_{15}(x_{2},x_{3},x_{4})=x_{1},\\
\vspace{2mm}\psi_{15}(x_{1},x_{3},x_{4})=x_{2},\\
\vspace{2mm}\psi_{15}(x_{2},x_{3},x_{1}^{*})=-x_{4}^{*},\\
\vspace{2mm}\psi_{15}(x_{3},x_{4},x_{1}^{*})=-x_{2}^{*},\\
\vspace{2mm}\psi_{15}(x_{2},x_{4},x_{1}^{*})=x_{3}^{*},\\
\vspace{2mm}\psi_{15}(x_{1},x_{3},x_{2}^{*})=-x_{4}^{*},\\
\vspace{2mm}\psi_{15}(x_{3},x_{4},x_{2}^{*})=-x_{1}^{*},\\
\vspace{2mm}\psi_{15}(x_{1},x_{4},x_{2}^{*})=x_{3}^{*}.
\end{array}\right.
\end{array}
\begin{array}{llllllllll}
   \left\{\begin{array}{l}
\vspace{2mm}\widetilde{\Delta}^{15}(x_{1}^{*})=x_{2}^{*}\wedge x_{4}^{*}\wedge x_{3}^{*},\\
\vspace{2mm}\widetilde{\Delta}^{15}(x_{2}^{*})=x_{1}^{*}\wedge x_{4}^{*}\wedge x_{3}^{*},\\
\vspace{2mm}\widetilde{\Delta}^{15}(x_{1})=x_{2}\wedge x_{3}^{*}\wedge x_{4}^{*},\\
\vspace{2mm}\widetilde{\Delta}^{15}(x_{2})=x_{1}\wedge x_{3}^{*}\wedge x_{4}^{*},\\
\vspace{2mm}\widetilde{\Delta}^{15}(x_{3})=x_{1}\wedge x_{4}^{*}\wedge x_{2}^{*}+x_{1}^{*}\wedge x_{2}\wedge x_{4}^{*},\\
\vspace{2mm}\widetilde{\Delta}^{15}(x_{4})=x_{1}\wedge x_{2}^{*}\wedge x_{3}^{*}+x_{1}^{*}\wedge x_{3}^{*}\wedge x_{2},\\
\vspace{2mm}\widetilde{\Delta}^{15}(x_{3}^{*})=\widetilde{\Delta}^{15}(x_{4}^{*})=0,\\
\vspace{2mm}\widetilde{\Delta}^{15}(x_{5})=\widetilde{\Delta}^{15}(x_{5}^{*})=0.\\
\end{array}\right.
\end{array}
\end{equation}
\begin{equation}
\begin{array}{llll}
   \left\{\begin{array}{l}
\vspace{2mm}\psi_{16}(x_{2},x_{3},x_{4})=-x_{2},\\
\vspace{2mm}\psi_{16}(x_{1},x_{3},x_{4})=x_{1},\\
\vspace{2mm}\psi_{16}(x_{1},x_{2},x_{3})=x_{3},\\
\vspace{2mm}\psi_{16}(x_{1},x_{2},x_{4})=-x_{4},\\
\vspace{2mm}\psi_{16}(x_{2},x_{3},x_{2}^{*})=x_{4}^{*},\\
\vspace{2mm}\psi_{16}(x_{3},x_{4},x_{2}^{*})=x_{2}^{*},\\
\vspace{2mm}\psi_{16}(x_{2},x_{4},x_{2}^{*})=-x_{3}^{*},\\
\vspace{2mm}\psi_{16}(x_{1},x_{3},x_{1}^{*})=-x_{4}^{*},\\
\vspace{2mm}\psi_{16}(x_{3},x_{4},x_{1}^{*})=-x_{1}^{*},\\
\vspace{2mm}\psi_{16}(x_{1},x_{4},x_{1}^{*})=x_{3}^{*},\\
\vspace{2mm}\psi_{16}(x_{1},x_{2},x_{3}^{*})=-x_{3}^{*},\\
\vspace{2mm}\psi_{16}(x_{1},x_{3},x_{3}^{*})=x_{2}^{*},\\
\vspace{2mm}\psi_{16}(x_{2},x_{3},x_{3}^{*})=-x_{1}^{*},\\
\vspace{2mm}\psi_{16}(x_{1},x_{2},x_{4}^{*})=x_{4}^{*},\\
\vspace{2mm}\psi_{16}(x_{1},x_{4},x_{4}^{*})=-x_{2}^{*},\\
\vspace{2mm}\psi_{16}(x_{2},x_{4},x_{4}^{*})=x_{1}^{*}.
\end{array}\right.
\end{array}
\begin{array}{llllllll}
   \left\{\begin{array}{l}
\vspace{2mm}\widetilde{\Delta}^{16}(x_{1}^{*})=x_{1}^{*}\wedge x_{4}^{*}\wedge x_{3}^{*},\\
\vspace{2mm}\widetilde{\Delta}^{16}(x_{2}^{*})=x_{2}^{*}\wedge x_{3}^{*}\wedge x_{4}^{*},\\
\vspace{2mm}\widetilde{\Delta}^{16}(x_{3}^{*})=x_{1}^{*}\wedge x_{3}^{*}\wedge x_{2}^{*},\\
\vspace{2mm}\widetilde{\Delta}^{16}(x_{4}^{*})=x_{1}^{*}\wedge x_{2}^{*}\wedge x_{4}^{*},\\
\vspace{2mm}\widetilde{\Delta}^{16}(x_{1})=x_{1}\wedge x_{3}^{*}\wedge x_{4}^{*}+x_{2}^{*}\wedge x_{3}^{*}\wedge x_{3}\\
\vspace{2mm}\hspace{1.5cm}+x_{4}^{*}\wedge x_{2}^{*}\wedge x_{4},\\
\vspace{2mm}\widetilde{\Delta}^{16}(x_{2})=x_{3}\wedge x_{3}^{*}\wedge x_{1}^{*}+x_{4}^{*}\wedge x_{3}^{*}\wedge x_{2}\\
\vspace{2mm}\hspace{1.5cm}+x_{4}\wedge x_{1}^{*}\wedge x_{4}^{*},\\
\vspace{2mm}\widetilde{\Delta}^{16}x_{3})=x_{1}^{*}\wedge x_{2}^{*}\wedge x_{3}+x_{2}^{*}\wedge x_{4}^{*}\wedge x_{2}\\
\vspace{2mm}\hspace{1.5cm}+x_{4}^{*}\wedge x_{1}^{*}\wedge x_{1},\\
\vspace{2mm}\widetilde{\Delta}^{16}(x_{4})=x_{1}^{*}\wedge x_{4}\wedge x_{2}^{*}+x_{1}\wedge x_{1}^{*}\wedge x_{3}^{*}\\
\vspace{2mm}\hspace{1.5cm}+x_{2}\wedge x_{3}^{*}\wedge x_{2}^{*},\\
\vspace{2mm}\widetilde{\Delta}^{16}(x_{5}^{*})=\widetilde{\Delta}^{16}(x_{5})=0.
\end{array}\right.
\end{array}
\end{equation}
\begin{equation}
\begin{array}{llll}
   \left\{\begin{array}{l}
\vspace{2mm}\psi_{17}(x_{1},x_{2},x_{3})=x_{1},\\
\vspace{2mm}\psi_{17}(x_{1},x_{2},x_{1}^{*})=-x_{3}^{*},\\
\vspace{2mm}\psi_{17}(x_{2},x_{3},x_{1}^{*})=-x_{1}^{*},\\
\vspace{2mm}\psi_{17}(x_{1},x_{3},x_{1}^{*})=x_{2}^{*}.
\end{array}\right.
\end{array}
\begin{array}{llllllllll}
   \left\{\begin{array}{l}
\vspace{2mm}\widetilde{\Delta}^{17}(x_{1}^{*})=x_{1}^{*}\wedge x_{3}^{*}\wedge x_{2}^{*},\\
\vspace{2mm}\widetilde{\Delta}^{17}(x_{1})=x_{1}\wedge x_{2}^{*}\wedge x_{3}^{*},\\
\vspace{2mm}\widetilde{\Delta}^{17}(x_{2})= x_{3}^{*}\wedge x_{1}^{*}\wedge x_{1},\\
\vspace{2mm}\widetilde{\Delta}^{17}(x_{3})=x_{1}^{*}\wedge x_{2}^{*}\wedge x_{1},\\
\vspace{2mm}\widetilde{\Delta}^{17}(x_{2}^{*})=\widetilde{\Delta}^{17}(x_{3}^{*})=\widetilde{\Delta}^{17}(x_{4}^{*})=0,\\
\vspace{2mm}\widetilde{\Delta}^{17}(x_{5}^{*})=\widetilde{\Delta}^{17}(x_{4})=\widetilde{\Delta}^{17}(x_{5})=0.
\end{array}\right.
\end{array}
\end{equation}

\end{theorem}

\begin{proof}
Apply  Lemma \ref{lem:4}, and Eq\eqref{eq:semiprod}, and Theorem \ref{thm:666} and  Theorem \ref{thm:888}.
\end{proof}

\subsection*{Acknowledgements}
The first named author was supported in part by the Natural
Science Foundation (11371245) and the Natural
\normalsize
Science Foundation of Hebei Province (A2018201126).

\end{document}